\documentclass[11pt]{amsart}

\usepackage{amsmath, amssymb, amsthm, amscd}
\usepackage{url}
\usepackage{graphicx}
\usepackage[hidelinks,pagebackref,pdftex]{hyperref}
\usepackage{color}
\usepackage{import}
\usepackage[margin=3cm]{geometry}
\usepackage{caption}
\usepackage{bbold}

\usepackage{amsmath}
\usepackage{old-arrows}
\usepackage{comment}

\newcommand{\Z}{\mathbb{Z}}
\newcommand{\N}{\mathbb{N}}

\newcommand{\calC}{\mathcal{C}}

\newcommand{\eg}{{\em e.g.}}

\newtheorem{innercustomthm}{Theorem}
\newenvironment{customthm}[1]
  {\renewcommand\theinnercustomthm{#1}\innercustomthm}
  {\endinnercustomthm}

\DeclareMathOperator{\reg}{reg}

\DeclareMathOperator{\Mod}{Mod}
\DeclareMathOperator{\poly}{poly}

\usepackage{marginnote}
\long\def\@savemarbox#1#2{\global\setbox#1\vtop{\hsize\marginparwidth 
  \@parboxrestore\tiny\raggedright #2}}
\renewcommand*{\backref}[1]{}
\renewcommand*{\backrefalt}[4]{
  \ifcase #1
  [No citations.]
  \or [#2]
  \else [#2]
  \fi }

\AtBeginDocument{%
   \def\MR#1{}
}

\numberwithin{equation}{section}
\theoremstyle{plain}
\newtheorem{theorem}[equation]{Theorem}

\newtheorem{lemma}[equation]{Lemma}
\newtheorem{corollary}[equation]{Corollary}
\newtheorem{proposition}[equation]{Proposition}

\newtheorem*{namedtheorem}{\theoremname}
\newcommand{\theoremname}{testing}

\theoremstyle{definition}
\newtheorem{definition}[equation]{Definition}

\newtheorem{remark}[equation]{Remark}

\title[Hardness of computation of quantum invariants on restricted topologies]{Hardness of computation of quantum invariants on 3-manifolds with restricted topology}
\author{Henrique Ennes}\thanks{INRIA d'Universit\'e C\^ote d'Azur, \texttt{henrique.lovisi-ennes@inria.fr}.}
\author{Cl\'ement Maria}\thanks{INRIA d'Universit\'e C\^ote d'Azur, \texttt{clement.maria@inria.fr}.}

\begin{document}

\begin{abstract}
Quantum invariants in low-dimensional topology offer a wide variety of valuable invariants about knots and 3-manifolds, presented by explicit formulas that are readily computable.
Their computational complexity has been actively studied and is tightly connected to topological quantum computing. 
In this article, we prove that for any 3-manifold quantum invariant in the Reshetikhin-Turaev model, there is a deterministic polynomial time algorithm that, given as input an arbitrary closed 3-manifold $M$, outputs a closed 3-manifold $M'$ with the same quantum invariant, such that $M'$ is hyperbolic, contains no low genus embedded incompressible surface, and is presented by a strongly irreducible Heegaard diagram. 
Our construction relies on properties of Heegaard splittings and the Hempel distance. 
At the level of computational complexity, this proves that the hardness of computing a given quantum invariant of 3-manifolds is preserved even when severely restricting the topology and the combinatorics of the input. 
This positively answers a question raised by Samperton \cite{samperton2023topological}.
\end{abstract}

\maketitle
\section{Introduction}
\label{sec:typesetting-summary}
{\em Quantum invariants} are topological invariants defined using tools from physics, explicitly, from topological quantum field theories (TQFTs). 
These invariants have become of interest for modeling phenomena in condensed matter physics~\cite{Arouca_2022, Fradkin_2024}, topological quantum computing~\cite{KITAEV20032}, and experimental mathematics~\cite{chen2018quantum,MariaR21}, where many deep conjectures remain open.
Thanks to their diversity and discriminating power to distinguish between non-equivalent topologies, they have also played an important role in the constitution of censuses of knots and 3-manifolds~\cite{burton:LIPIcs.SoCG.2020.25}. 
The invariants are constructed from the data of a fixed algebraic object, called a {\em modular category}, and a topological support, and take the form of a partition function, whose value depends solely on the topological type of the support and not on its combinatorial presentation.
Generally, quantum invariants are defined for presentations of either knots or 3-manifolds, although there are types of invariants, such as the Reshetikhin-Turaev, whose definitions naturally encompass both objects.

The complexity of the exact and approximate computation of these invariants has attracted much interest, particularly in connection with quantum complexity classes.
Most non-trivial quantum invariants turn out to be \#\texttt{P}-hard to compute and, sometimes, even \#\texttt{P}-hard to approximate~\cite{alagic2014quantum,freedmanLW2002topological,kuperberg2009hard} within reasonable precision. 
This is the case of the Jones polynomial for knots~\cite{freedmanLW2002topological,kuperberg2009hard} and the Turaev-Viro invariant of 3-manifolds associated with the Fibonacci category~\cite{alagic2014quantum, freedmanLW2002topological}. 
On the positive side, polynomial time {\em quantum} algorithms exist for computing weak forms of approximations~\cite{aharonovJL2009polynomial,aradL2010quantum,freedmanLW2002topological} and efficient {\em parameterized} algorithms have been designed~\cite{burton:LIPIcs.SoCG.2018.18,BurtonMS18,MAKOWSKY2003742,Maria21,MariaS20} leading to polynomial time algorithms on certain families of instances. 
For the latter, the topology of the input knot or 3-manifold plays a central role in measuring the computational complexity of the problem, either directly with running times depending strongly on some topological parameter~\cite{delaney2024algorithmtambarayamagamiquantuminvariants, MariaS20}, or indirectly where {\em simple topologies} guarantee the existence of simple combinatorial representations~\cite{huszar_et_al:LIPIcs.SoCG.2019.44, Huszár_2022, MariaP19} that can be in turn algorithmically exploited.

There are other instances, however, where the hardness of computing the invariants is preserved even if the topology and combinatorics of the input are restricted. In~\cite{kuperberg2009hard}, Kuperberg shows that, for certain quantum invariants that are {\em hard to approximate} on {\em links}, the hardness is preserved when restricted to {\em knots}. In a follow-up work, Samperton~\cite{samperton2023topological} proves that if computing a quantum invariant is hard for all input diagrams of any knot, then the computation remains hard when restricting the input to {\em hyperbolic} knots given by diagrams with a minimal number of crossings. In this article, we follow a similar path to both \cite{kuperberg2009hard} and \cite{samperton2023topological}, this time for quantum invariants of 3-manifolds, by proving the hardness of computing invariants of irreducible presentations and hyperbolic manifolds.


We follow the strategy of Samperton~\cite{samperton2023topological} which consists of using Vafa's theorem \cite{vafa1988toward} to complicate, in polynomial time, the topological structure of the input without changing the invariant. Nonetheless, while Samperton's process involves adding extra crossings to the knot diagrams, we increase the \textit{Hempel} \textit{distance} of a \textit{Heegaard diagram} of some 3-manifold. Although there exists an extensive catalog of algorithms to increase Hempel distances \cite{evans2006high, hempel20013, ido2014heegaard,johnson2013non, lustig2009high, qiu2015heegaard, yoshizawa2014high}, to the best of our knowledge, our work is the first to 1. explicitly compute the involved complexities, ensuring polynomial time; and, 2. keep some 3-manifold invariant constant throughout the process.

The main result is expressed in Theorem \ref{th: properties}, whose precise statement can be found in Section \ref{sec: main results}. 
Here and throughout the paper, we denote by $\Sigma_g$ the \textit{closed surface of genus} $g$ (which is unique up to homeomorphism/diffeomorphism) with some fixed orientation and assume $g\geq 2$.
When referring to a general compact surface (potentially with boundary), we use $\Sigma$.

\begin{theorem}
Let $\calC$ be a modular category and $M$ a closed 3-manifold represented by a Heegaard diagram $(\Sigma_g,\alpha, \beta)$ of complexity $m$. 
There is a deterministic algorithm that constructs, in time $O(\poly(m,g))$ and uniformly on the choice of $\calC$, a strongly irreducible Heegaard diagram $(\Sigma_{g+1},\alpha', \beta')$ representing a hyperbolic 3-manifold $M'$ that shares with $M$ the Reshetikhin-Turaev invariant over $\calC$.
Moreover, for a fixed choice of $k\in \N$, $M'$ has no embedded orientable and incompressible surface of genus at most $2 k$.
\end{theorem}

\begin{remark}  
    Hyperbolicity has historically been used as both a simplifying structure and an intermediate step for algorithms on 3-manifold, see for example \cite{kuperberg2019algorithmic,scull2021homeomorphismproblemhyperbolicmanifolds}. 
    Similarly to \cite{samperton2023topological} for knots, our result proves that hyperbolicity is of no help for the computational complexity of the quantum invariant.
    On the other hand, when producing hard instances of 3-manifolds in computational topology---\eg, in complexity reduction~\cite{agol2006computational, bachman2017computing} or the construction of combinatorially involved manifolds~\cite{huszar_et_al:LIPIcs.SoCG.2023.42}---it is common to produce \emph{Haken} 3-manifolds with low genus incompressible surfaces (generally, tori).
    It is an important open question to understand the hardness of computation for non-Haken 3-manifold. Our result shows that the computational complexity of quantum invariants is preserved even when getting rid of low-genus incompressible surfaces. 
\end{remark}

The paper is divided into four parts: a review of background material (Section \ref{sec: background}), the demonstration of some auxiliary algorithms (Section \ref{sec: algorithms}), the proof our main result (Section \ref{sec: main results}), and some illustration of its computational consequences for the hardness of computing quantum invariants (Section \ref{sec: computational consequences}).

\section*{Acknowledgement} We are grateful to Lukas Woike, Alex He, and Nicolas Nisse for helpful discussions on Vafa's theorem, triangulations of Heegaard splittings, and general complexity theory, respectively. We also thank the anonymous reviewers for their thoughtful comments and constructive feedback, which greatly contributed to improving the clarity, rigor, and quality of this paper. This work has been partially supported by the ANR project ANR-20-CE48-0007 (AlgoKnot).

\section{Background material}\label{sec: background}
In the following review, we assume acquaintance with the basic ideas from geometric topology, such as boundaries, compactness, homeomorphisms, (free) homotopies, isotopies, and manifolds. 
For these topics, we refer the reader to \cite{schultens2014introduction} and \cite{singer2015lecture}. 
Moreover, we shall use, without properly defining, some well-known concepts belonging to the theories of curves and surfaces, which can be found in \cite{farb2011primer}. 

\subsection{Curves in surfaces}
A \textit{simple closed curve} in a surface $\Sigma$ is a connected component of the image of a proper embedding $S^1\hookrightarrow \Sigma$. 
We will often refer to a simple closed curve by the terms \textit{closed curve} or just \textit{curve}. 
A \textit{multicurve} is a finite collection of disjoint properly embedded simple curves in $\Sigma$.
We denote by $\#\gamma$ its number of connected components. 
Whenever possible, we distinguish simple curves from multicurves by using Greek letters to represent the latter.

A curve in the surface will be called \textit{essential} if not homotopic to a point (or equivalently, if it does not bound a disk on the surface), a puncture, or a boundary component. 
A multicurve is essential if all its components are essential.
Similarly, a properly embedded \emph{arc} (i.e., with endpoints in $\partial \Sigma$) is essential if it does not cobound a disk on $\Sigma$ with a component of $\partial\Sigma$.
Unless otherwise stated, all arcs, curves, and multicurves will be assumed essential.


We will often be interested not in a curve $s$, but in the equivalence classes of $s$ up to ambient \emph{isotopies} in $\Sigma_g$, $[s]$.
Being essential is preserved under isotopies, so we naturally extend the definition of essential curves to their isotopy classes.
For two curves $s,t$ in some surface $\Sigma$, their \textit{geometric intersection number} is defined as the minimal number of their intersection points up to isotopy, that is $i(s,t) = \min\{|s'\cap t'|: s'\in [s],t'\in [t]\}$.
Two curves in a surface $\Sigma_g$ are isotopic if and only if they are (free) homotopic \cite{farb2011primer}.

The \textit{curve graph} of a closed surface $\Sigma_g$, $C(\Sigma_g)$, is the graph whose vertices are isotopy classes of essential curves, and any two vertices $[s]$ and $[t]$ are connected by an edge if and only if $i(s,t)=0$. The usual graph distance $d$ defines a metric on the vertices $C(\Sigma_g)$ that is interpreted as $d([s],[t])=n$ implying the existence of a sequence of essential curves $r_0,r_1,\dots,r_n$ with $r_0\in [s]$, $r_n\in [t]$ and $r_i\cap r_{i+1}=\emptyset$ for $0\leq i < n$. The definitions related to the curve graph can be naturally extended to curves by identifying the elements of an isotopy class to the same vertex; in particular, $d(r,s)=0$ is equivalent to $r$ and $s$ being isotopic. 


\begin{figure}
\centering
\includegraphics[width=0.7\textwidth]{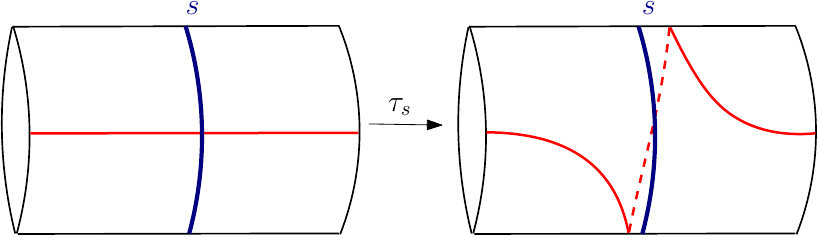}
\caption{Illustration of the action of a Dehn twist about a curve $s$ (blue) on some curve transversal to it (red). Note that the homeomorphism is only different from the identity on a regular neighborhood of $s$.}\label{fig: dehn twist}
\end{figure}

We recall that for each surface $\Sigma_g$, its \textit{mapping class group}, $\Mod(\Sigma_g)$, is the group of orientation-preserving homomorphism $\Sigma_g\to \Sigma_g$ up to isotopies. Its canonical action on the surface conserves the geometric intersection number between curves \cite{farb2011primer}, acting, therefore, isometrically on $(C(\Sigma_g),d)$ by the induced map $\phi\cdot [s]=\phi_*([s])$. The mapping class groups always contain the \textit{Dehn twists}, homeomorphisms $\tau_s:\Sigma_g\to \Sigma_g$ defined by cutting off a local neighborhood of the (multi)curve $s$ of $\Sigma_g$ and gluing it back with a $2\pi$ counterclockwise twist, determined by the orientation of the surface (Figure \ref{fig: dehn twist}). In particular, for each $g\geq 2$, $\Mod(\Sigma_g)$ is generated as a group by Dehn twists about the $3g-1$ recursively defined \textit{Lickorish curves} in $\Sigma_g$ (Figure \ref{fig: lickorish}) \cite{lickorish1964finite}.  In this paper, we will always assume that a homeomorphism class in $\Mod(\Sigma_g)$ is of the form $\phi=\tau^{n_r}_{s_r}\circ \dots \circ \tau^{n_1}_{s_1}$, where $n_i\in \Z$ and $s_i$ are Lickorish curves for all $1\leq i \leq r$.

\begin{figure}
\centering
\includegraphics[width=0.5\textwidth]{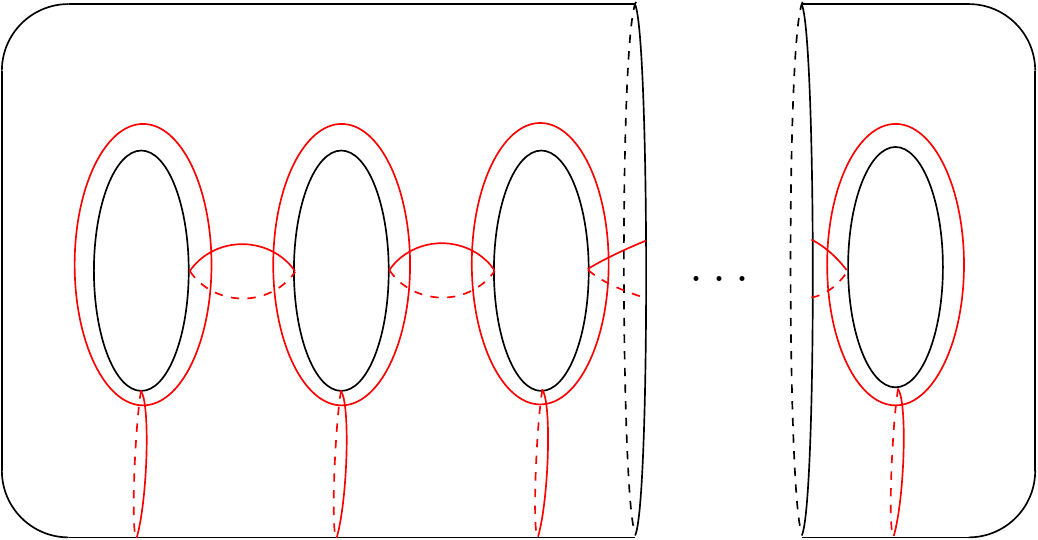}
\caption{}\label{fig: lickorish}
\end{figure}

\subsection{Curves on handlebodies and Heegaard splittings}
An essential multicurve $\gamma \subset C(\Sigma_g)$ is a \textit{(full) system} if no two components are isotopic to each other and $\Sigma_g- \gamma$ is a union of $\#\gamma-g+1$ punctured spheres (here the minus sign indicates cutting the surface along $\gamma$). The minimum number of connected components that a full system may have is $g$, in which case $\Sigma_g- \gamma$ is a $2g$-punctured sphere and $\gamma$ is called a \textit{minimum} system. On the other extreme, a system $\gamma$ is \textit{maximum} or a \textit{pants decomposition} when $\#\gamma=3g-3$ and $\Sigma_g-\gamma$ is the union of $2g-2$ thrice punctured spheres, also known as \textit{pairs of pants} (see Figure \ref{fig: seams}). When every connected component of a system $\gamma$ is isotopic to a connected component of another system $\gamma'$, we say that $\gamma$ is \textit{contained} in $\gamma'$ and denote that by $\gamma\subseteq \gamma'$. We note that a minimum system $\gamma$ can always be (non-uniquely) extended to a pants decomposition $\rho\supseteq \gamma$, refer to Theorem \ref{th: minimal to pants}.

The genus $g$ \textit{handlebody} {constructed through a full system} $\gamma$ in the surface $\Sigma_g$ is the 3-manifold
\begin{equation*}
    V_\gamma = \Sigma_g\times [0,1]\cup_{\gamma \times \{0\}} 2\text{-handles}\cup 3 \text{-handles}
\end{equation*}
built by attaching the 2-handles along the curves $\gamma$ in $\Sigma\times \{0\}$ and then filling any resulting $S^2$ boundary component with 3-handles. By construction $\partial V_\gamma=\Sigma_g$. Because the curves in the system $\gamma$ are assumed to be essential, each of its components will bound (non-trivial) compression disks in $V_\gamma$: they will be \textit{meridians} of the handlebody. There are, however, many other meridians in $V_\gamma$, as it will be implied by the following definition.

\begin{definition}[Disk graph and equivalent systems]\label{def: disk graph}
    Let $V_\gamma$ be a handlebody constructed over $\Sigma_g$. 
    Then the disk graph of $V_\gamma$, $K_\gamma$, is the subgraph of $C(\Sigma_g)$ whose vertices represent (isotopy classes of) meridians of $V_\gamma$. 
    We say that two full systems, $\gamma$ and $\gamma'$, (potentially with $\#\gamma\neq \#\gamma'$) are \textit{equivalent} if  $\gamma'\subset K_\gamma$ and $\gamma\subset K_{\gamma'}$. 
\end{definition}
 Note that $\gamma$ is equivalent to $\gamma'$ if and only if they define the same handlebody.
 In particular, if $\gamma\subset \rho$, $\gamma$ and $\rho$ are equivalent.

In a seminal work, Hempel \cite{hempel20013} studied the metric $d$ of the disk graph canonically inherited from $C(\Sigma_g)$. 
This inspires the next definition. 
Here and throughout, whenever $A,B\subset C(\Sigma)$ and $r$ is a curve in $\Sigma$, we naturally extend the metric $d$ to $d(r,A)=\min \{d(r,s):s\in A\}$ and $d(A,B)=\min\{d(s,t):s\in A,t\in B\}$.
We always assume a curve $s$ on a handlebody $V_\gamma$ to be fully contained in $\partial V_\gamma$. 

\begin{definition}[Diskbusting curves]
    An essential curve $s$ on a handlebody $V_\gamma$ is said to be \textit{diskbusting} if $d(s,K_\gamma)\geq 2$, that is, if $s$ intersects all meridians of $V_\gamma$.
\end{definition}
\cite{starr1992curves} provides a combinatorial condition to verify if a curve on a handlebody is diskbusting, which we quote using the more appropriate language of \cite{yoshizawa2014high}. Before, however, we will need a definition.

\begin{definition}[Seams and seamed curves]
    An arc in a pair of pants ${P}$ is called a \emph{seam} if it has endpoints on two distinct components of $\partial P$. 
    A curve $s$ in a surface $\Sigma_g$ with a pants decomposition $\rho$ is said to be \emph{seamed} for $\rho$ if, for every pair of pants component $P$ of $\Sigma_g- \rho$, $s\cap P$ has at least one copy of each of the three types (up to isotopy) of seams in $P$. 
\end{definition}

\begin{theorem}[Theorems 1 of \cite{starr1992curves}, Theorem 4.11 of \cite{yoshizawa2014high}]\label{th: diskbusting}
Let $s$ be a curve on the handlebody $V_\gamma$. Then $s$ is diskbusting in $V_\gamma$ if and only if there is a pants decomposition $\rho$ equivalent to $\gamma$ such that $s$ is seamed for $\rho$.
\end{theorem}

A closed 3-manifold $M$ is said to have a \emph{Heegaard splitting} if it is the union of two handlebodies intersecting only at their common boundary. 
It is well-known \cite{saveliev2011lectures} that there exists, for every closed 3-manifold $M$, a tuple $(\Sigma_g,\alpha,\beta)$ called a \emph{Heegaard diagram (of genus $g$)}, where $\alpha$ and $\beta$ are two full systems with the same cardinality in $\Sigma_g$, and a Heegaard splitting $M = V_\alpha \cup_{\Sigma_g} V_\beta$. 
Note, on the other hand, that neither the Heegaard diagram nor the splitting is unique: for example, isotopies to $\alpha$ or $\beta$ yield the same splitting, whereas given a diagram $(\Sigma_g, \alpha,\beta)$, one can define another splitting for the same manifold, this time with diagram $(\Sigma_{g+1},\alpha\cup\{c\},\beta\cup \{c'\})$, where $c$ and $c'$ are curves with $i(c,c')=1$ fully contained in the extra handle. 
This last process, known as \textit{stabilization}, is topologically equivalent to directly summing a copy of $S^3$ to the original manifold $M$.

Some properties of 3-manifolds can be read straight-up from their Heegaard splittings. For example, a splitting $(\Sigma_g, \alpha, \beta)$ is called \textit{irreducible} if there is no essential curve $s$ in $\Sigma_g$ that is a meridian for both $V_\alpha$ and $V_{\beta}$. Haken's lemma \cite[Theorem 6.3.5]{schultens2014introduction} implies that reducible closed 3-manifolds cannot have irreducible Heegaard splittings. Similarly, a splitting $(\Sigma_g, \alpha , \beta)$ is \textit{strongly irreducible} if there are no two disjoint essential curves $a$ and $b$ in $\Sigma_g$ such that $a$ is a meridian of $V_{\alpha}$ and $b$ is a meridian of $V_{\beta}$. Every strongly irreducible splitting is irreducible, but the converse is not true (Haken manifolds provide a list of counterexamples \cite{schultens2014introduction}). 

Given two handlebodies $V_\alpha$ and $V_\beta$, we define their \textit{Hempel distance} by $d(K_\alpha,K_\beta)$. In \cite{hempel20013}, Hempel argued that this distance could be seen as a measure of the complexity of a Heegaard splitting $V_\alpha \sqcup_{\Sigma_g} V_\beta$, which is translated into the following theorem.

\begin{theorem}\label{th: hempel}
Let $(\Sigma_g,\alpha,\beta)$ be a Heegaard diagram of distance $d(K_\alpha,K_\beta)=k$. Then
    \begin{itemize}
    \item $k\geq 1$ if and only if $(\Sigma_g,\alpha,\beta)$ is irreducible;
    \item $k\geq 2$ if and only if $(\Sigma_g,\alpha,\beta)$ is strongly irreducible;
    \item if $k\geq 3$, then $M$ is hyperbolic;
    \item $M$ has no orientable and incompressible embedded surface of genus smaller than $2k$.
\end{itemize}
\end{theorem}

\begin{proof}
The first two points are jointly known as Casson-Gordon rectangular conditions and were first published by Kobayashi \cite{kobayashi1988casson}. The third is a consequence of the already mentioned work by Hempel \cite{hempel20013} and the main result in \cite{thompson1999disjoint}, followed by the proof of the Geometrization Conjecture due to Perelman \cite{perelman2003finiteextinctiontimesolutions, perelman2003ricciflowsurgerythreemanifolds}. The last is due to Hartshorn \cite{hartshorn2002heegaard}.
\end{proof}

\begin{remark}
    Scharlemann and Tomova \cite{scharlemann2006alternate} proved that if $V_{\alpha'}\cup_{\Sigma_{g'}}V_{\beta'}$ is Heegaard splitting of genus smaller than $k/2$ then $V_\alpha \cup_{\Sigma_g} V_\beta$ is \emph{isotopic} to $V_{\alpha'}\cup_{\Sigma_{g'}}V_{\beta'}$, potentially after finitely many stabilizations. In particular, if $k>2g+2$, the splitting is of minimum genus. Unfortunately, as we will see in the proof of Theorem \ref{th: properties}, our algorithm is not polynomial time as function of $k$, which means that it does not imply an efficient reduction to a minimal genus splitting for every input.
\end{remark}

Our proof of Theorem \ref{th: properties} will mainly consist of increasing the Hempel distance so that the hypotheses of Theorem \ref{th: hempel} are satisfied. 
For such we will amply use the next two theorems due to Yoshizawa.

\begin{theorem}[Theorem 5.8 of \cite{yoshizawa2014high}]\label{th: yoshi}
    Consider the full systems of curves $\alpha$ and $\beta$ in $\Sigma_g$ and $n=\max\{1, d(K_{\alpha},K_{\beta})\}$. Let $d_i=d(K_i,s)$ for $i=\alpha,\beta$, assume $d_i\geq 2$ and $d_\alpha+d_\beta-2> n$. Then for any $k\in \Z^+$
    \begin{equation}\label{eq: yoshizawa inequality}
        \min(k,d_\alpha+d_\beta-2)\leq d(K_{\alpha}, K_{\tau^{k + n + 2}_s(\beta)})\leq d_\alpha+d_\beta.
    \end{equation}
\end{theorem}

\begin{theorem}[Theorem 6.2 of \cite{yoshizawa2014high}]\label{th: yoshi 2}
    Let $\gamma=\{c_1,\dots,c_g\}$ be full in $\Sigma_g$ and $\rho$ a pants decomposition containing $\gamma$. Suppose $s$ is seamed for $\rho$ and define the multicurve $\tau^2_s(\gamma)$ of components $d_1,\dots, d_g$. Then $d(K_{\gamma}, \tau_{d_g}^2\circ \dots \circ \tau_{d_1}^2(c_1))\geq 3$.
\end{theorem}

Note that Theorem \ref{th: yoshi} hints towards unboundedly increasing the Hempel distance of a splitting $(\Sigma_g,\alpha,\beta)$, provided a curve $s$ of distance at least 3 from $K_\alpha$ or $K_\beta$ (see the proof of Proposition \ref{prop: launching from diskbusting}). In this sense, Theorem \ref{th: yoshi 2} is complementary, as it describes how to construct such a curve $s$ from a diskbusting input.

\subsection{Quantum invariants for 3-manifolds}\label{sec: quantum}
We will not review the technical construction of TQFTs here, referring the interested reader to \cite{turaev2010quantum}. 
For our purposes, it is enough to know that, for a fixed choice of modular category $\mathcal{C}$ (again, refer to \cite{turaev2010quantum} for the definition), the TQFT associates to every closed 3-manifold $M$ a complex scalar known as its \emph{Reshetikhin-Turaev} (\emph{RT}) invariant. 
The RT invariant can be given as a function of a Heegaard diagram $(\Sigma_g,\alpha,\beta)$ of $M$ and is denoted by $\langle M\rangle^{RT}_{\mathcal{C}}$ or $\langle (\Sigma_g,\alpha,\beta)\rangle^{RT}_{\mathcal{C}}$, depending on whether we want to emphasize the manifold or the diagram. 

The algebraic structure imposed by the modular category sets some constraints on the quantum invariants. 
The next theorem, for example, is already somewhat folklore in the literature (see for example \cite{muller2023dehn}) since it implies that non-homeomorphic 3-manifolds may share RT invariants. 
It can be deduced from Theorem 5.1 of \cite{etingof2002vafa} and \cite{salvatore2001frameddiscsoperadsequivariant, turaev2010quantum}.

\begin{theorem}[Vafa's theorem for 3-manifold TQFTs]\label{th: vafa}
    Let $\mathcal{C}$ be a modular category. Then there is an $N\in \Z^+$, depending only on $\mathcal{C}$, such that, for all $k\in \Z$ and every curve $s$ in $\Sigma_g$, $\langle(\Sigma_g, \alpha, \beta)\rangle_\mathcal{C}^{RT} = \langle(\Sigma_g, \alpha, \tau^{kN}_s(\beta))\rangle^{RT}_\mathcal{C}$.  
\end{theorem}
For a fixed modular category, we call the integer $N$ the category's \textit{Vafa's constant}.

\subsection{Relevant data structure}\label{sec: data structure}
An embedded graph $G=(V,E)$ in a surface $\Sigma$ defines a \textit{cellular complex} for $\Sigma$ if $\Sigma- G$ is a union of open disks, which we call \textit{faces}.
Note that this implies that any boundary component of $\Sigma$ fully lies within some set of edges in $E$. 
The \textit{dual graph} of $G$ is another graph embedded in $\Sigma$ defined by assigning a vertex to each face of $\Sigma- G$ and an edge between the vertices if and only if the corresponding faces are separated by an edge in $E$.
As a data structure, we represent a cellular embedding by the lists of faces, their incident edges, and vertices, allowing us to reconstruct $\Sigma$ by gluing the appropriate pieces.
Using this structure, one can compute the dual graph of the cellular embedding in linear time on the number of faces. 

A \textit{(generalized) triangulation} $T=(V,E)$ \textit{of a surface} $\Sigma$ is a cellular complex where all faces are bounded by exactly 3 edges.
We denote the number of triangles in a triangulation by $|T|$; note that $|T|=O(|E|)$. 
We will most often consider \textit{oriented triangulations} by giving an orientation to each triangle consistent with the orientation of the surface; we orient a triangle by imposing an order to its vertices as in the right-hand rule. 
We say that a triangulation $T'$ of a surface $\Sigma$ is a \textit{subtriangulation} of another triangulation $T$ of $\Sigma$, denoting $T \leq T'$, if the graph of $T$ is embedded in $T'$ (i.e., each vertex of $T$ is a vertex of $T'$ and each edge of $T$ is a union of edges in $T'$). 
If $T$ and $T'$ are oriented, we also require the order of the vertices to be the same. 
Every face of $T'$ is naturally contained in a face of $T$.

Using a triangulation $T=(V,E)$ of a surface $\Sigma$, we can describe an arc, curve, or multicurve $s$ in $\Sigma$ lying fully within the edges $E$ by the list of the edges in $E\cap s$ (the red curve in Figure \ref{fig: overlay}). 
We will call this list an \textit{edge list representation} of the curve $s$, denote it by $E_T(s)$, and say that the number of edges in $E\cap s$ is the \textit{edge complexity}, $\|E_T(s)\|$.
We can always assume that $\|E_T(s)\|\leq |T|$. 

\begin{figure}
\centering
\includegraphics[width=0.8\textwidth]{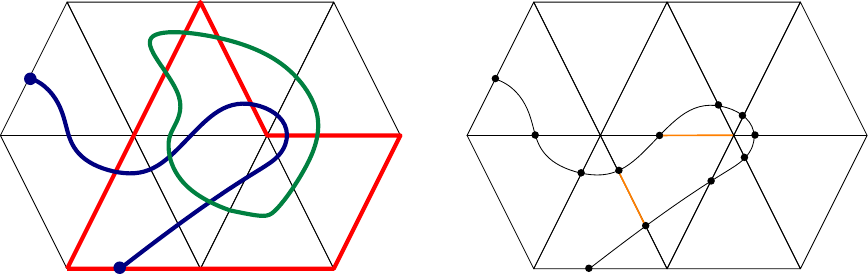}
\caption{Left: a triangulation of the disk with a curve represented by edge list (red), a standard curve (green), and a normal arc (blue). Right: the overlay graph of the normal arc, with some ports (orange) highlighted.}\label{fig: overlay}
\end{figure}

For any fixed triangulation $T$ of the surface $\Sigma$, there are always, however, isotopy classes of curves in $\Sigma$ not representable as a subset of edges of $T$. 
To deal with this hindrance, we say, for a fixed triangulation $T=(V,E)$ of the surface $\Sigma$, that a curve $s$ is \textit{standard} (with respect to $T$) if it intersects $T$ only transversely and at edges (green curve in Figure \ref{fig: overlay}).
If $s$ is standard, we may represent it as an \textit{intersection word} $I_T(s)$, taking $E$ as an alphabet and traversing $s$ along some arbitrary direction and, whenever we meet an edge $e\in E$, we append $e$ to $I_T(s)$ if $e$ is crossed according to the orientation of $\Sigma$ and $e^{-1}$ otherwise. 
While the edge representation is well-defined at the curve level, intersection words are defined only up to isotopies inside the triangulation's faces.
Furthermore, isotopy classes of $I_T(s)$ are closed under cyclic permutations and taking inverses.
The \textit{complexity of an intersection word} $I_T(s)$, $\|I_T(s)\|$, is its length (i.e., the number of edges of $T$ intersected by $s$, counted with multiplicity). 
If $\gamma$ is a multicurve, we let $I_T(\gamma)$ be the set of $\#\gamma$ intersection words representing each component. 
Similarly, the complexity of a Heegaard diagram $(\Sigma_g,\alpha,\beta)$, where $\alpha$ and $\beta$ are standard multicurves, equals $\|I_T(\alpha)\|+\|I_T(\beta)\|$. Standard arcs are treated accordingly.

 A standard curve, multicurve or arc $s$ is \textit{normal} if no intersection word $I_T(s)$ contains a substring of form $ee^{-1}$ or $e^{-1}e$ where $e\in E$ (the blue curve in Figure \ref{fig: overlay}). 
 That is, when normal, the intersections of $s$ and the faces of the triangulation are arcs connecting distinct edges of each triangle.
 The union of these arcs with the original triangulation defines a cellular embedding $T_s$ which we call the \emph{overlay graph} of $s$ (right side of Figure \ref{fig: overlay}).
 By extending the definition of subtriangulation to general cellular embeddings, we see that $T\leq T_s$. 
 Therefore, the edges of the overlay graph are either the intersections of $s$ and the faces of $T$, or (potentially empty) subdivisions of the edges of $T$, called \emph{ports}. In particular, for any curve (or multicurve) $s$, we note that the number of faces of $T_s$ is $|T|+\|I_T(s)\|-1$. 

 \begin{remark}\label{rm: normal complexity}
    When a curve is normal for a triangulation, its isotopy class is fixed by the number of times it intersects each labeled edge of $E$ \cite{schaefer2002algorithms}, meaning that it can be described through a vector in $\mathbb{N}^{|E|}$ called the curve's \emph{normal coordinates}. It is, therefore, natural to use the normal coordinates as a data structure for curves in surfaces \cite{bell2016polynomial, bell2015recognising,schaefer2002algorithms, schaefer2008computing}, with their complexity given by the amount of space, in binary, necessary to represent the vector. Although this gives an exponential gain in storage space, efficient algorithms on this compressed encoding are unnecessary to guarantee the complexity results of Section \ref{sec: computational consequences}.
\end{remark}

\begin{remark}
    Normal curves can be understood as the smaller dimension version of the normal surface theory, widely used by Haken and follow-up works to analyze, for example, the problems of unknot detection and orientability of embedded surfaces in 3-manifolds \cite{haken1961theorie, matveev03-algms}.
\end{remark}

\begin{lemma}\label{lm: compute overlay}[Section 3.1 of \cite{erickson2012tracing}]
    Suppose that $s$ is a normal arc, curve, or multicurve with respect to some triangulation $T$, with $\|I_T(s)\|=m$. Then, one can compute, in time $O(m+|T|)$, the overlay graph $T_s\geq T$.
\end{lemma}

\section{Algorithms for curves in surfaces and handlebodies}\label{sec: algorithms}

\subsection{Converting between representations of curves}
In Section \ref{sec: data structure}, we saw two representations of curves in a surface: edge lists
and intersection words. 
While some topological operations such as cutting along a curve are easier to implement using edge list representations (one needs only to delete the edges crossing $E_T(s)$ from the dual graph of $T$ to cut along $s$), others, such as doing Dehn twists (Theorem \ref{th: normal coordinates}) are more suited to curves intersecting the triangulation transversely.
Therefore, as a first step in proving Theorem \ref{th: properties}, it will be convenient to have polynomial time procedures to convert from one representation to the other. 

First, we describe an algorithm to transform curves represented by edge lists into intersection words. 
For an edge list $E_T(s)$ not in a boundary of an oriented surface $\Sigma$, there are two choices of normal curves isotopic to $s$ created by slightly displacing it either to the left or to the right of the edges (with respect to the orientation of $\Sigma$ and some arbitrary orientation of $s$), we call them \textit{twins born from} $s$. 
Given a choice of twin for $s$, say left, one can compute its intersection word by traversing $s$ and appending the letters representing adjacent edges coming from the left side of the triangulation graph when embedded in $\Sigma$. 
If, however, $s$ lies either partially or fully on the boundary of $\Sigma$, we can only consistently displace it to one side, which can be determined in time $O(|T|)$. 
This algorithm does not change the triangulation. 
For convenience, we state this argument as the following proposition.

\begin{proposition}\label{prop: triangulation to normal}
    Suppose $E_T(s)$ is an edge list representation of a curve $s$ in $\Sigma$. Then there is an algorithm to compute an intersection word $I_T(s)$ of complexity $O(|T|)$ in time $O(|T|)$.
\end{proposition}


Now we turn to the other direction, that is, going from an intersection to an edge representation. 
Although by Lemma \ref{lm: compute overlay} one can, in time $O(|T|+\|I_T(\gamma)\|)$, trace $T_\gamma$, the overlay graph is not a triangulation of the surface. 
Luckily, we may always extend it to a triangulation $T'\geq T_\gamma \geq T$ as shown in the proof of the next proposition.

\begin{proposition}\label{prop: normal to triangulation}
    Suppose that $s$ is a normal arc, curve, or multicurve with respect to some triangulation $T$, with $\|I_T(s)\|=m$. 
    Then, there exists an algorithm that constructs, in time $O(m+|T|)$, a new triangulation $T' \geq T_s$ of size $O(m+|T|)$.
\end{proposition}

\begin{proof}
   Assume $s$ is a normal curve, the multicurve case is treated similarly. 
   Apply Lemma \ref{lm: compute overlay} to compute $T_s$ and note that, although it has $|T|+m-1$ faces, at most $m-1$ of them are not triangles (they can be rectangles, pentagons, or hexagons). 
   Suppose one of them is a hexagon, we can transform it into a union of triangles by choosing a vertex and adding 3 edges connecting it to its non-adjacent vertices, increasing the number of faces in $T_s$ by 3. 
   The rectangles and pentagons can be treated similarly.
   In total, one needs to add to $T_s$, at most, $3m-3$ new edges to construct a triangulation $T'\geq T_s$.
\end{proof}

\subsection{Basic operations on curves, handlebodies, and Heegaard splittings}
Before proceeding with the main results, we will establish algorithms for some basic operations, namely applying a Dehn twist, extending minimal to maximum systems, constructing a diskbusting curve for a minimal system, and stabilizing a Heegaard splitting. 
The first of these results is only a minor modification of \cite{schaefer2008computing}, whereas the proof of the second in the context of curves in surfaces is mostly due to \cite[Theorem 7.1]{verdiere2007optimal}.

\begin{theorem}\label{th: normal coordinates}
Suppose $t$ and $s$ are normal (multi)curves for some fixed triangulation $T$ of $\Sigma_g$, given through intersection words $I_T(s)$ and $I_T(t)$ with $m=\max\{\|I_T(s)\|,\|I_T(t)\|\}$. 
Then an intersection word for $\tau^k_{s}(t)$ for all $k\in \Z$, with $\|I_T(\tau^k_{s}(t))\|\leq |k|m^3$, can be computed in time $O(|k|m^3)$.
\end{theorem}

\begin{proof}
   Before we start, some notation. If $t$ is a standard curve in $T=(E,V)$ and $p$ is a labeled point of $t\cap E$, we denote by $I_T(t,p)$ the intersection word of $t$ starting at $p$, with some fixed choice of traversing direction not explicit in the notation. 
   Given any intersection word $I_T(t)$ and a fixed $p_i\in t\cap E$, one can compute $I_T(t,p)$ in time $O(\|I_T(t)\|)$ by cyclic permuting and reversing the input word until $p$ is the first letter and the orientation agrees with the expected orientation of $I_T(t,p)$.
   
   Now, for the actual proof, we will isotope the curves $t$ and $s$ inside each face of the triangulation so they intersect in a particular way. 
   If two sides of the triangle are identified, there is, up to an isotopy inside the triangle's face, only one normal arc, meaning that $t$ and $s$ can be made disjoint. 
   If no two sides of the triangle are identified, we convention $t$ to intersect the leftmost port of $T_s$ if it turns left (with respect to the triangle's orientation) after the intersection, otherwise, it intersects the rightmost port (Figure \ref{fig: dehn}). 
   In this way, $t$ may only intersect $s$ if it does a left turn followed by a right one or vice versa. 
   Consider some region where the two curves intersect, e.g., $t$ does a right turn to intersect an edge $e$ and then a left turn to intersect an edge $f$ (see Figure \ref{fig: dehn}). 
   Order the intersection points of $e$ and $s$ from right to left following the orientation induced by the tracing direction of $t$ and the surface's orientation. 
   If the intersection happens because $t$ does a right turn followed by a left turn, order the intersections of $e$ from left to right. 
   We compute, in time $O(m^2)$, the intersection word $I_T(s, p_i)$ launching from each of the points $p_i, i=1,\dots,|s\cap e|,$ and following the original direction of tracing. 
   We then append the word $I_T(s,p_1)^kI_T(s,p_2)^k\dots I_T(s,{|s\cap e|})^k$ in between the letters $e$ and $f$ on $I_T(t)$. 
   In total, we must append these strings, of length $k|s\cap e|m\leq km^2$, at most $m$ times.
   The output word, $I_T(\tau_s^k(t))$, although representing a standard curve, might not represent a normal one, having some occurrences of $ee^{-1}$ or $e^{-1}e$. 
   To ensure normality, we cyclically reduce the word in linear time on $\|I_T(\tau^k_s(t))\|$. 
   For, it takes $O(\|I_T(\tau^k_s(t))\|)$ time to (non-cyclic) reduce the word by eliminating adjacent occurrences of $e$ and $e^{-1}$ and at most half of this time to eliminate occurrences of $e$ and $e^{-1}$ in the first and last letters of the resulting word, respectively.
\end{proof}

 \begin{figure}
    \centering
    \includegraphics[width=0.5\textwidth]{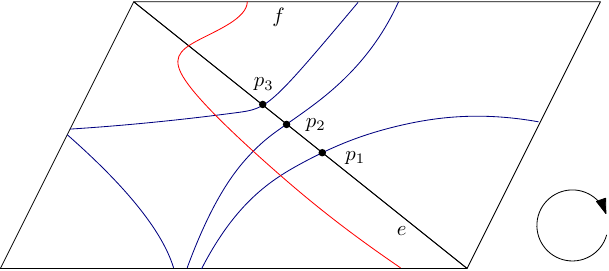}
    \caption{The curve $t$ (in red) performing a right turn to intersect the edge $e$ followed by a left turn to intersect the edge $f$, traced according to the proof of Theorem \ref{th: normal coordinates}. The black triangulation corresponds to $T$. Labeled intersection points ($p_1$,$p_2$, and $p_3$) of $s$ and $e$ as in the proof of the theorem are also highlighted. The arrow signals the orientation of $T$.}\label{fig: dehn}
\end{figure}

\begin{theorem}\label{th: minimal to pants}
    Suppose $\gamma$ is a minimal system of edge curves in $\Sigma_g$. Then one can compute, in time $O(g|T|)$, a triangulation $T'\geq T$ of size $O(m|T|)$ and a pants decomposition $\rho$ in the edges of $T'$, with $\|E_{T'}(\rho)\|=O(g|T|)$, such that $\rho$ contains $\gamma$.
\end{theorem}
Before proving Theorem \ref{th: minimal to pants}, we will first establish the following lemma.
 \begin{lemma}\label{lm: cut punctured sphere}
    Suppose that $T$ is a triangulation of a $n$-punctured sphere, $S^{2}_n$, $n\geq 4$, with boundary components $\beta=\{b_1,\dots,b_n\}$ of edge complexity at most $m$. 
    Then one can find, in time $O(|T|)$, an intersection word of a normal curve $d$ of complexity $O(|T|)$ that separates $S^{2}_n$ into a pair of pants and $S^2_{n-1}$.
\end{lemma}

\begin{proof}
    Appealing to Proposition \ref{prop: triangulation to normal}, we compute an intersection word representation of the boundary components, $I_T(\beta)=\{I_T({b_1}),\dots I_T({b_n})$\}; because these are boundary components, there is no choice of twin to be made.
    Greedily choose one $b\in \beta$ and one face $t$ ajacent to $b$. 
    Using breadth-first search, we find, in time $O(|T|)$, the shortest path $p$ in the dual graph starting from $t$ to any a face adjacent $E_T(b')$, $b'\in \beta\backslash b$---note that $t$ might be adjacent to either an edge or a vertex of $E_T(b)$ and $E_T(b')$.
    We can encode $p$ as a list of faces of $T$.
    In time $O(\|E_T(b)\|)=O(|T|)$, we update $p$ by iteratively removing from this list all but the last face adjacent to $b$.
    By construction, $p$ intersects each edge of the triangulation at most once and does not intersect any face of the triangulation adjacent to a curve in $\beta$, except for exactly one face adjacent to $E_T({b})$ and one face adjacent to $E_T({b'})$.
    In particular, $\|I_T(p)\|=O(|T|)$.
    We define the curve $d$ as the concatenation of $b,p$, and $b'$, i.e., $I_T(d)=I_T(b)I_T(p)I_T(b')I_T(p)^{-1}$ (Figure \ref{fig: proof eric}, note that this operation is sometimes called a \textit{band sum of} $b$ \textit{and} $b'$),  where we permute the words $I_T(b)$ and $I_T(b')$ so they start and end with letters representing edges on the respective faces of the triangulation intersected by $p$. 
    Moreover, because $p$ intersects only one such face for $b$ and $b'$, no reductions are necessary to make $I_T(d)$ normal. 
    We notice that, since it is separating, $d$ is isotopic to neither $b$ nor $b'$ in $S^{2}_n$.
\end{proof}

\begin{figure}
    \centering
    \includegraphics[width=0.35\textwidth]{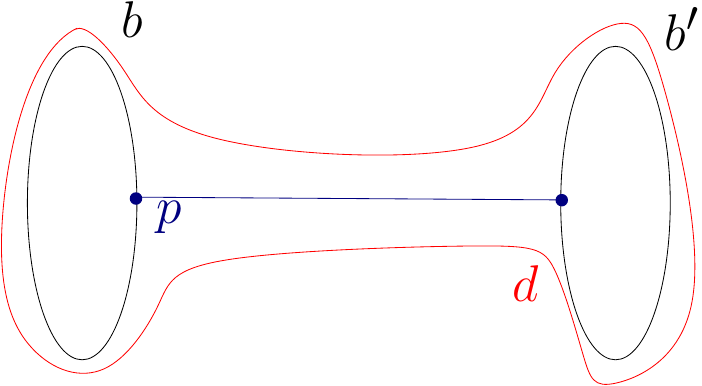}
    \caption{Diagram representing the curves $b$, $b'$ (black) and $d$ (red) and the arc $p$ (blue) defined in the proof of Lemma \ref{lm: cut punctured sphere}.}\label{fig: proof eric}
    \end{figure}

\begin{proof}[Proof of Theorem \ref{th: minimal to pants}]
    Recall that $\Sigma_g- \gamma$ is homeomorphic to $S_{2g}^2$, whose boundary components we call $\beta_0$. 
    Therefore, we proceed by recursively applying Lemma \ref{lm: cut punctured sphere} to construct $2g-3$ curves $d_1,\dots, d_{2g-3}$ such that $S_{2g}^2- d_1$ is a pair of pants and a copy of $S_{2g-1}^2$, $S^2_{2g-1}- d_2$ is a pair of pants and copy of $S^2_{2g-2}$, etc.
    At the end of the $i$-th step, we change the triangulation as in Proposition \ref{prop: normal to triangulation} so that the curve $d_i$ is in the edges of a cellular embedding of $\Sigma_g$ (we do not, however, transform the overlay graphs $T_{d_i}$ into triangulations yet). 
    Each tracing iteration takes linear time on $|T|$. 
    The induction's base step is a direct application of the lemma; for the $i$-th step, we let $\beta_i=\{b_1,\dots,b_{2g-i}\}$ be the boundary components of the punctured sphere $S_{2g-i}^2$. Each $b_j\in \beta_i$ is either a component of $\beta_0$ or a separating curve $d_j, 1\leq j\leq i-1,$ constructed in a previous step. We then apply Lemma \ref{lm: cut punctured sphere} to this multicurve $\beta_i$. Although we complicate the dual graph at each step by adding edges to represent $d_i$ and cutting the surface along them, the breadth-first iterations need only to consider paths on a subgraph of the current dual graph with at most $O(|T|)$ nodes, for $p$ does not cross faces adjacent to curves that it does not connect and $d_i$ are all separating. In total, we look for $2g-3$ shortest paths, giving a time complexity of $O(g\times( m+|T|))=O(g|T|)$. We finally transform the overlay graph $T_\rho$ into a triangulation in time $O(g|T|)$.
    
    It remains to show that the curves $d$ are meridians. Assume they are not essential. If $b$ and $b'$ are not isotopic when regluing $\Sigma_g- \beta_{2g-3}$ to form $\Sigma_g$, giving an orientation to the curves involved, the path $b\circ p\circ (b')^{-1}\circ  p^{-1} $ would be contractible (with $\circ$ representing the usual path concatenation operation), meaning that $b$ and $b'$ are homotopic in $\Sigma_g$ and, consequently, isotopic.  
    If, on the other hand, $b$ and $b'$ are isotopic in the reglued $\Sigma_g$, by the construction in Lemma \ref{lm: cut punctured sphere}, they are also isotopic to a single component of $\gamma$ and $p$ is a non-homology trivial cycle $\Sigma_g$ with $|b\cap p|=i(p,b)=1$. Using $|b\cap p|$ as a base point, the loop $d= b\circ p \circ b^{-1}\circ p^{-1}$ is contractible on $\Sigma_g$ only in the ignored case where $g=1$, when $b$ and $p$ commute in the fundamental group (refer to \cite[Chapter 1]{hatcher2002algebraic} for the definition of fundamental group). 
    Finally, we refer to the construction in Section 2.2 of \cite{johnson2006notes} for the fact that $p$ bounds a compressing disk on $V_\gamma$.
\end{proof}

\begin{theorem}\label{th: get diskbusting}
    Suppose $\gamma$ is a normal minimal system in $\Sigma_g$ with respect to a triangulation $T$, with $\|I_T(\gamma)\|=m$. Then there is a deterministic algorithm that outputs, in time $O((m+|T|)g^2)$, a disksbuting curve $s$ for the disk graph $K_\gamma$, which is normal to $T$ and is represented by intersection word of complexity $O(g|T|)$.
\end{theorem}
\begin{proof}
    Use Proposition \ref{prop: normal to triangulation} to get an edge list representation of $\gamma$ in a new triangulation $T'$ of complexity $O(m+|T|)$ and then Theorem \ref{th: minimal to pants} for an equivalent pants decomposition $\rho\supseteq\gamma$ of complexity $O(g|T'|)$ when represented by an edge list in another triangulation $T''$ of $\Sigma_g$. 
    We will construct a multicurve $\sigma$, seamed for the pants decomposition $\rho$, that intersects each component exactly twice. 

    For each component $r$ of $\rho$, select two edges contained in $E_T(r)$ to be points of intersection with $\sigma$. 
    For every connected component $P$ of $\Sigma_g-\rho$ (i.e., $P$ is a pair of pants), we connect each intersection point of a boundary component of $P$ to other intersection points in the two other components of $\partial P$.
    To avoid $\sigma$ of self-intersecting, we separately draw in each pair of pants a seam at a time, using breadth-first search in the dual graph as in the proof of Theorem \ref{th: minimal to pants}, recording the intersection words, and changing the dual graph so that no seam can cross a previously traced seam. 
    Even though we change the dual graph, making it more complicated at each drawing of a seam, because the seams are local within pairs of pants, each application of breadth-first search considers only $O(|T''|)$ nodes and edges. 
    At the end of the process, we have a multicurve $\sigma$ that, although seamed for the pants decomposition $\rho$, may have up to $g+1$ connected components. 
    It therefore remains to modify $\sigma$ so it has just a single component. 

    \begin{figure}
    \centering
    \includegraphics[width=0.7\textwidth]{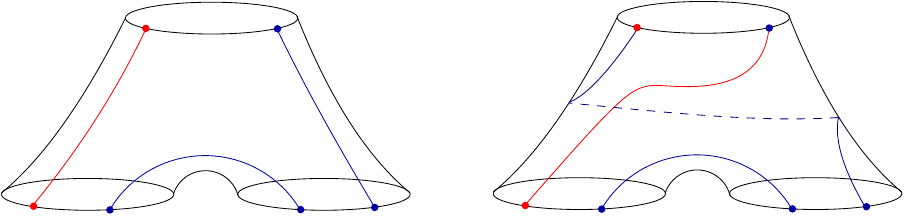}
    \caption{Left: the three types of seams, up to isotopy, on a pair of pants. Two connected components (a red and a blue one) of the multicurve $\sigma$ are highlighted. Right: a surgery to connect the two distinct components.}\label{fig: seams}
    \end{figure}

    Whenever there still are disconnected components in $\sigma$, there exists a pair of pants intersected by at least two distinct components of $\sigma$, refer to Figure \ref{fig: seams}. 
    We can then do the surgery on the right side of Figure \ref{fig: seams} to connect the two components. 
    Using breadth-first, this takes a total time of $O(g|T''|)$ and yields a connected curve $s$ with intersection word $I_{T''}(s)$. 
    Because, by construction, $s$ is seamed for each pair of pants from $\Sigma_g- \rho$, by Theorem \ref{th: diskbusting}, $s$ is diskbusting for $\gamma$. 
    Finally, we compute an intersection word of $s$ with respect to $T$ in time $O(|T''|)$ by deleting the edges in $T''\backslash T$. 
    Note that $s$ is already normal for $T$ as, by construction, $s$ is standard with respect to $T''\geq T'\geq T$ and, because all paths are shortest, no cyclic reductions are possible.
\end{proof}

\begin{lemma}\label{lm: stabilization}
    Suppose $(\Sigma_g,\alpha,\beta)$ is a Heegaard diagram, with $\alpha$ and $\beta$ given as intersection words in a triangulation $T$ of $\Sigma_g$, where we let $m=\max\{\|I_T(\alpha)\|,\|I_T(\beta)\|\}$. 
    Then one can compute, in time $O(m)$, a triangulation $T'$ and new multicurves $\alpha',\beta'$ which represent a stabilization $(\Sigma_{g+1},\alpha',\beta')$ of the original diagram, with $O(|T'|)=O(|T|)$ and $\|I_{T'}(\alpha')\|,\|I_{T'}(\beta')\|= O(m)$.
\end{lemma}
\begin{proof}
    Choose any triangular face $t$ of $T$ of sides $e,f,$ and $g$, as in the left side of Figure \ref{fig: stabilization}.
    As in the figure, we add three vertices and nine new edges to $t$, which will now have an inner triangle of sides $\widetilde{e},\widetilde{f},$ and $\widetilde{g}$ on its center, ultimately retriangulating the face $t$ into a union of triangles $\widetilde{t}$. 
    Call $\widetilde{T}$ the triangulation $T$ with $t$ substituted for $\widetilde{t}$. 
    Note that $\alpha$ and $\beta$ can be both isotoped not to intersect the inner triangle. Moreover, one can compute $I_{\widetilde{T}}(\alpha)$ and $I_{\widetilde{T}}(\beta)$ by iterating over the intersection words $I_T(\alpha)$ and $I_T(\beta)$ and inserting the intersections with the edges in $\widetilde{T}\backslash T$ in between any adjacent pair of letters that represent an arc in $t$. 
    For example, referring again to Figure \ref{fig: stabilization}, we substitute any occurrence of the pair $ef$ in $I_T(\alpha)$ and $I_T(\beta)$ for $ex_1x_2 f$.

    Now let $T^2- t'$ be the torus with a triangle removed and endow it with a fixed-size triangulation (right side of Figure \ref{fig: stabilization}). 
    Let $a$ and $b$ be two essential curves intersecting each other once in $T^2- t'$, normal to the endowed triangulation, and given as intersection words.
    The triangulation $T$ given by identifying $\widetilde{t}$ to $t'$ represents $\Sigma_{g+1}$. 
    Moreover, the splitting $(\Sigma_{g+1},\alpha',\beta')$ where $\alpha'=\alpha\cup \{a\}$ and $\beta'=\beta\cup \{b\}$ is a stabilization of $(\Sigma_g,\alpha,\beta)$.
\end{proof}

    \begin{figure}
    \centering
    \includegraphics[width=0.7\textwidth]{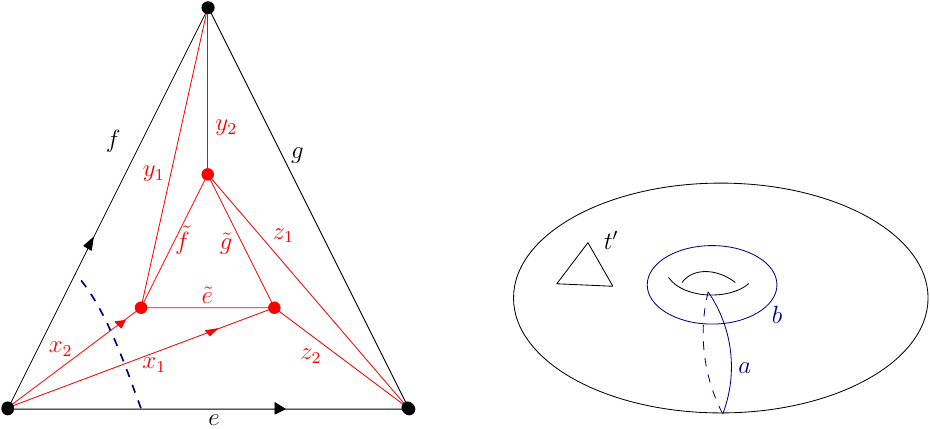}
    \caption{The diagrams referred to in the proof of Lemma \ref{lm: stabilization}. Left: the triangle $t$ (black) and its subtriangulation $\widetilde{t}$ (red). The subarc $ef$ is shown in dashed blue. The orientation of the edges $e$, $f$, $x_1$, and $x_2$ are shown by the small arrows. Right: the surface $T^2-t'$ and the curves $a $ and $b$ (blue).}\label{fig: stabilization}
    \end{figure}




\section{Restricting the topology of 3-manifolds}
\label{sec: main results}
Given a Heegaard diagram $(\Sigma_g,\alpha, \beta)$ and a fixed $k\in \Z^+$, the algorithm of Theorem \ref{th: properties} uses Theorems \ref{th: yoshi} and \ref{th: vafa} to construct a new diagram $(\Sigma_{g+1}, \alpha,\tau^{n}_s(\beta))$ with Hempel distance at least $k$, without altering the associated RT-invariant in the process, where $n$ is some integer multiple of the associated Vafa's constant and $s$ is a curve distant enough from both $K_\alpha$ and $K_\beta$. 
The main challenge of the algorithm comes, however, from building such a curve $s$.
We address the problem by first computing, through Proposition \ref{prop: launching from diskbusting},
two curves, say $s_\alpha$ and $s_\beta$, of distance at least $k$ from $K_\alpha$ and $K_\beta$, respectively. 
We then use these two curves in Proposition \ref{prop: get large distance from s_alpha and s_beta} to construct a single curve $s$, whose distances to $K_\alpha$ and $K_\beta$ satisfy the hypotheses of Theorem \ref{th: yoshi}. 

We establish Proposition \ref{prop: launching from diskbusting} incrementally, first using Theorem \ref{th: diskbusting} to find a curve $s$ diskbusting to $K_\gamma$, increasing $d(K_\gamma,s)$ to at least 3 using Lemma \ref{lm: distance 3 curve}, and finally making the distance bigger than $k$.

\begin{lemma}\label{lm: distance 3 curve}
    Let $\gamma$ be a minimal system in $\Sigma_g$, normal to a triangulation $T$ of $\Sigma_g$, and represented by an intersection word $I_T(\gamma)$ with $\|I_T(\gamma)\|=m$. One can compute, in time $O((gm|T|)^{9})$, an intersection word of a normal curve $s'$ of complexity $O((gm|T|)^9)$, with $d(K_\gamma, s)\geq 3$.
\end{lemma}
\begin{proof}
    Suppose $\gamma$ has components $\{c_1,\dots,c_g\}$. Construct, using Theorem \ref{th: diskbusting}, a curve $s$, seamed for a pants decomposition $\rho\supseteq\gamma$, and of complexity $O(g|T|)$. Once this $s$ is constructed, we can use Theorem \ref{th: yoshi 2} to find the desired curve $s'$.
    
    Explicitly, first compute an intersection word of complexity $O(2\max\{g|T|,m\}^3)$ for the multicurve $\tau^2_{s}(\gamma)$, of components $d_1,\dots,d_n$, using Theorem \ref{th: normal coordinates}. Once again, use Theorem \ref{th: normal coordinates} to compute $s' = \tau^{2}_{d_g}\circ \dots \circ \tau^{2}_{d_1}(c_1)$ in time $O(16\max\{g|T|,m\}^9)$. Theorem \ref{th: yoshi 2} guarantees that such a curve $s'$ has the desired property.
\end{proof}


\begin{proposition}\label{prop: launching from diskbusting}
     Let $\gamma$ be a minimal system in $\Sigma_g$, normal with respect to some triangulation $T$. Then, for a fixed $3\leq k\in \Z^+$, one may compute, in time $c_2^{O(k^{c_1})}\log k$, where $c_1=\log 3$ and $c_2=O(gm|T|k)$, the intersection word of a curve $s$, normal to $T$ and of complexity $c_1^{O(k^{c_2})}$, for which $d(K_\gamma, s) > k$.
\end{proposition}
\begin{proof}
Start using Lemma \ref{lm: distance 3 curve} to find an intersection word of a curve $s_0$ with $d(K_\gamma,s_0)\geq 3$. 
If $k=3$ we are done, so assume otherwise. 
Define $M=\|I_T(s_0)\|$; recall that $M= O((gm|T|)^9)$. 
Let $c$ be any connected component of $\gamma$ and recursively define $s_{i+1} = \tau^{k_{i}+3}_{s_i}(c)$, where $k_i = 2^{i+1}+2$. 
We claim that, for any $\ell \in \Z^+$, $d(K_{\gamma}, s_\ell)\geq k_{\ell-1}$. We prove this inducting on $\ell$: for $\ell=1$, $k_0=4$ and, by Theorem \ref{th: yoshi} with $\alpha=\beta=\gamma$ and $n=\min\{d(K_\gamma,K_\gamma),1\}=1$, we have that
\begin{equation*}
    d(K_{\gamma},s_1)=d(K_\gamma, \tau^{k_0+n+2}_{s_0}(c))\geq  d(K_{\gamma}, K_{\tau_{s_0}^{k_0+n+2}(\gamma)})\geq 4.
\end{equation*}
Now assume $d(K_\gamma, s_\ell)\geq 2^{\ell}+2$. Again, by Theorem \ref{th: yoshi}, 
\begin{equation*}
    d(K_\gamma, s_{\ell+1}) = d(K_\gamma, \tau_{s_\ell}^{k_\ell+3}(c)) \geq d(K_\gamma, K_{\tau_{s_\ell}^{k_\ell+3}(\gamma)}) \geq 2d(K_\gamma, s_\ell)-2\geq 2(2^{\ell}+2)-2=2^{\ell+1}+2,
\end{equation*}
so, by induction, $d(K_{\gamma}, s_\ell)\geq k_{\ell-1}$.
Setting the total number of iterations at $\ell=\lceil \log (k - 2) \rceil$ and $s=s_\ell$, we have that $d(K_\gamma, s)\geq k$.

We now estimate the total computational time and the output's complexity of the algorithm.
First, we note that, by Theorem \ref{th: normal coordinates}, $\|I_T(s_{i+1})\|\leq (k_{i}+3)\|I_T(s_{i})\|^3$. Recursively, this gives
\begin{equation*}
\begin{split}
    \|I_T(s)\| &\leq M^{3^\ell}(k_0+3)^{3^{\ell-1}}(k_1+3)^{3^{\ell-2}}\dots (k_{\ell-1}+3)\\
    & \leq M^{3^\ell}(2^1+5)^{3^{\ell-1}} (2^2+5)^{3^{\ell-2}}\dots (2^{\ell}+5)\\
    &\leq M^{3^\ell}(2^{\ell}+5)^{3^{\ell-1}} (2^{\ell}+5)^{3^{\ell-2}}\dots (2^{\ell}+5)\\
    & \leq M^{3^\ell}(2^{\ell}+5)^{3^{\ell-1}+ 3^{\ell-2}+\dots +3^0}\\
    & \leq M^{3^\ell}(2^{\ell}+5)^{\frac{1}{2}\times (3^{\ell}-1)}\\
    &\leq M^{3^{\log(k-2)+2}}(2^{\log(k-2)+2}+5)^{\frac{1}{2}\times (3^{\log(k-2)+2}-1)}\\
    &\leq M^{9\times 3^{\log(k-2)}}(4k-3)^{9/2\times 3^{\log(k-2)}}\\
    &\leq M^{9 (k-2)^{\log 3}}(4k)^{5 (k-2)^{\log 3}}\\
    &\leq M^{9k^{c_1}}(4k)^{5k^{c_1}}
\end{split}
\end{equation*}
where we used the geometric series in the fifth line, $\lceil \log(k-2)\rceil \leq \log(k-2) +2$ in the sixth line, the relation $a^{\log b}=2^{\log a \log b}= b^{\log a}$ for any real $a,b>1$ in the eighth line, and $c_1=\log 3$ in the last line. The time complexity for computing $I_T(s)$ can be (very coarsely) estimated at $O(\ell \times (k_{\ell-1}+3) \|I_T(s_{\ell-1})\|^3)=O(\log k \times \|I_T(s)\|)=O((gm|T|)^{81k^{c_1}}(4k)^{5k^{c_1}} \log k)$ by noting that $\|I_T(s_i)\|\leq \|I_T(s_{\ell-1})\|$ and $k_i\leq k_{\ell-1}=k$ for all $1\leq i \leq n$.
\end{proof}

We now show Proposition \ref{prop: get large distance from s_alpha and s_beta}, starting with the following technical lemma.

\begin{lemma}\label{lm: s is close, becomes big}
    Consider some full systems $\gamma$ and $\gamma'$ in the surface $\Sigma_g$ with $d(K_{\gamma},K_{\gamma'})\geq 4$ and a curve $s$ with $d(K_{\gamma},s)<2$. Then $d(K_{\gamma'},s)\geq 2$.
\end{lemma}

\begin{proof}
Assume the opposite, that is, $d(K_{\gamma'},s)<2$ and let the isotopy class $a\in K_\gamma$ be a minimizer of the distance between $s$ and $K_{\gamma}$ and the isotopy class $b \in K_{\gamma'}$ a minimizer of the distance between $s$ and $K_{\gamma'}$. Then
    \begin{equation*}
        d(a,b)\leq d(a,s)+d(s, b)< 2+2=4.
    \end{equation*}
Nonetheless, because $a\in K_{\gamma}$ and $b \in K_{\gamma'}$, $d(a,b)\geq d(K_{\gamma},K_{\gamma'})$, a contradiction.
\end{proof}

\begin{proposition}\label{prop: get large distance from s_alpha and s_beta}
Fix an integer $k\geq 4$. Consider some minimal systems $\alpha$ and $\beta$ in the surface $\Sigma_g$ for which $d(K_\alpha, K_\beta) = 0$. Then there is a curve $s$ and some full minimal system $\beta'$ (potentially $\beta'=\beta$) in $\Sigma_g$ such that
    \begin{equation}\label{eq: property}
    \begin{split}
         \min_{i=\alpha,\beta'}\{d(K_i,s)\}\geq 2,\max_{i=\alpha,\beta'}\{d(K_i,s)\}\geq k, d(K_{\alpha},s)+d(K_{\beta'},s)-2>n,\\
         \text{ and } \langle (\Sigma_g,\beta)\rangle^{RT}_\mathcal{C} = \langle (\Sigma_g,\beta')\rangle^{RT}_\mathcal{C},   
    \end{split}
    \end{equation}
where $n=\min\{d(K_{\alpha},K_{\beta'}),1\}$. Moreover, if $m = \max \{\|I_T(\alpha)\|,\|I_T(\beta)\|\}$, then $I_T({s})$ and $I_T({{\beta'}})$ will have complexity $Nk c_2^{O(k^{c_1})}$ and in similar time for any choice of Vafa's constant $N\in \Z^+$, where $c_1=\log 3$ and $c_2=O(gm|T|k)$.
\end{proposition}
\begin{proof}
We use Proposition \ref{prop: launching from diskbusting} to find two curves, $s_\alpha$ and $s_\beta$, such that $d(K_\alpha,s_\alpha)$ and $d(K_\beta,s_\beta)$ are larger than $k$. Note that if $s_\alpha$ is diskbusting in $K_{\beta}$, we are done, as we can let $s=s_\alpha$, $\beta'=\beta$, and
    \begin{equation*}
        d(K_\alpha, s)\geq k,\; d(K_{\beta'},s)=d(K_\beta,s_\alpha)\geq 2, \text{ and } d(K_\alpha,s)+d(K_{\beta'},s)-2>1.
    \end{equation*}
Therefore, assume $d(K_\beta,s_\alpha)<2$. 
    
By applying Theorem \ref{th: normal coordinates}, we compute an intersection word for $\widetilde{\beta}=\tau_{s_\beta}^{N(k+3)}(\beta)$. Note that $\|I_{T}(\widetilde{\beta})\|= O(Nk\|I_T(s_\beta)\|^3)= O(Nk(gm|T|)^{243k^{c_1}}(4k)^{15k^{c_1}})$ and, by Theorem \ref{th: yoshi}, $d(K_\beta,K_{\beta'})\geq k$. Moreover, because $\beta'$ was constructed by applying a power of Dehn twists multiple of $N$ to $\beta$, by Theorem \ref{th: vafa}, $\langle (\Sigma_g,\beta)\rangle^{RT}_\mathcal{C} = \langle (\Sigma_g,\beta')\rangle^{RT}_\mathcal{C}$.

    Pick any component $\widetilde{b}$ of ${\widetilde{\beta}}$. If $\widetilde{b}$ is diskbusting for $K_\alpha$, then $s=\widetilde{b}$ and $\beta'=\beta$ have the desired properties. In particular, notice that
    \begin{equation*}
        d(K_{\alpha},s)\geq 2, \;\; d(K_{\beta'},s)\geq d(K_{\beta},K_{\widetilde{\beta}})\geq k, \text{ and } d(K_{\alpha},s)+d(K_{\beta'},s)-2> 1.
    \end{equation*}
    If, however, $\widetilde{b}$ is not diskbusting for $K_\alpha$, it means that $d(K_{\alpha},K_{\widetilde{\beta}})\leq d(K_\alpha, \widetilde{b})<2$. Letting $s=s_\alpha$ and $\beta'=\widetilde{\beta}$ gives
    \begin{equation*}
        d(K_{\alpha},s)\geq k,\;\;d(K_{\beta'},s)\geq 2, \text{ and }d(K_{\alpha},s)+d(K_{\beta'},s)-2 \geq k\geq 3 > d(K_{\alpha},K_{\beta'}),
    \end{equation*}
    where the second inequality comes from $d(K_\beta,s_\alpha)< 2$ and $d(K_\beta,K_{\widetilde{\beta}})\geq k\geq 4$ applied to Lemma \ref{lm: s is close, becomes big}.
\end{proof}
Finally, we may once more combine the above result with Theorem \ref{th: yoshi} to prove our main reduction.

\begin{customthm}{1}\label{th: properties}
    Let $(\Sigma_g,\alpha,\beta)$ be a Heegaard diagram of a closed 3-manifold $M$ of complexity $m$ with respect to a triangulation $T$ of $\Sigma_g$. 
    Choose a modular category $\mathcal{C}$ with Vafa's constant $N$ and fix an integer $k\geq 4$.
    Then there is a set of three Heegaard diagrams, computed in time $O_k(\poly(g,m,|T|, N))$ (of degree depending on $k$), representing manifolds with RT invariant over $\mathcal{C}$ equal to $\langle M\rangle^{RT}_{\mathcal{C}}$, one of them guaranteed to be hyperbolic and with no embedded orientable and incompressible surface of genus at most $2k$.
\end{customthm}
\begin{proof}
We start by using Lemma \ref{lm: stabilization} to compute a stabilized splitting $(\Sigma_{g+1},\alpha',\beta')$; note that $d(K_{\alpha'},K_{\beta'})=0$. 
Construct two curves $s_\alpha$ and $s_\beta$ as in the proof of Proposition \ref{prop: get large distance from s_alpha and s_beta}. 
Let $\beta_1=\beta'$ and $s_1=s_\alpha$.

We proceed as in the proof of Proposition \ref{prop: get large distance from s_alpha and s_beta}, computing an intersection word for a new minimal system $\widetilde{\beta}=\tau^{N(k+3)}_{s_\beta}(\beta')$ and defining $\beta_2=\beta'$ and $\beta_3= \widetilde{\beta}$, $s_2$ as any component of $\widetilde{\beta}$, and $s_3=s_\alpha$.
By Proposition \ref{prop: launching from diskbusting}, for at least one $1\leq j \leq 3$
\begin{equation*}
  \min_{i=\alpha',\beta^j}\{d(K_i,s_j)\}\geq 2,\max_{i=\alpha',\beta^j}\{d(K_i,s_j)\}\geq k, d(K_{\alpha'},s_j)+d(K_{\beta_j},s_j)-2>n  
\end{equation*}
where $n = \max\{d(K_{\alpha'},K_{\beta_j}),1\}$.
Applying, once again, Theorems \ref{th: normal coordinates} and \ref{th: yoshi}, we can conclude that one of the splittings $(\Sigma_g,\alpha',\tau_{s_j}^{N(k+3)}(\beta_j))$ has a Hempel distance of, at least, $k$ and an intersection word for $\tau_{s_j}^{N(k+3)}(\beta_j)$ can be computed in time $$O(Nk(\max\{\|I_T(\beta_j)\|,\|I_T(s_j)\|\})^3)=O(N^4k^4(gm|T|)^{729k^{c_1}}(4k)^{45k^{c_1}}).$$ Theorem \ref{th: hempel} finishes the proof. 
\end{proof}

\begin{remark}
For analysis and algorithmic purposes, it is convenient to explicitly represent the metric in a hyperbolic 3-manifold, notably using a solution to {\em Thurston's gluing equations} in a triangulation~\cite[Chapter E.6]{BenedettiPetronioHyperbolicGeom}. It is however a difficult problem in general to prove the existence and actually build a {\em geometric} triangulation of a hyperbolic 3-manifold, i.e., admitting such solution. As of now, our reduction does not provide such a metric explicitly. 
\end{remark}

\section{Computational reduction for quantum invariants}\label{sec: computational consequences}
Theorem \ref{th: properties} gives a polynomial time algorithm to change a general closed 3-manifold into another manifold with very restricted topology without altering the RT invariant in the process. 
Therefore, the problems of either exactly computing or approximating the invariant of general 3-manifolds \emph{reduce}, in a Cook-Turing sense, to the problems of exactly computing or approximating the invariant when the manifolds are assumed to have the properties of Theorem \ref{th: properties}. 
Ultimately, this means that the \emph{hardness of computation} is not altered in this restricted topology scenario.

We illustrate this reduction by showing that \emph{value-distinguishing} approximations of the Reshetikhin–Turaev and the Turaev-Viro invariants are \#\texttt{P}-hard, even for manifolds with the properties of Theorem \ref{th: properties}, when we take $\mathcal{C}$ to be the category of representations of the quantum group $SO_r(3)$, for some prime $r\geq 5$ (in this case, one can use $N=4r$ for Vafa's constant).
We note that, by a value-distinguishing approximation, we mean the ability to determine whether the approximated quantity $c\in \mathbb{R}^+$ is $a>c$ or $b<c$ for any fixed $0<a<b$ where we assume, as a premise, that one of the two to hold. 
In particular, multiplicative approximations are value-distinguishing, although other less restrictive schemes also are \cite{kuperberg2009hard}.

\subsection{Reshetikhin–Turaev invariant}
For every Heegaard diagram $(\Sigma_g,\alpha,\beta)$, there exists an orientation-preserving homeomorphism $\phi:\Sigma_g\to \Sigma_g$ such that $\beta=\phi(\alpha)$, provided that $\alpha$ and $\beta$ are minimal \cite{farb2011primer, hatcher1980presentation}. 
Although not unique, this map is well-defined in $\Mod(\Sigma_g)$, so can be combinatorially described by a word on Lickorish generators.
Here and throughout, we do not differentiate the notation of the homomorphism $\phi$ from its equivalence class in $\Mod(\Sigma_g)$.

\begin{theorem}[\cite{alagic2014quantum, freedmanLW2002topological}]\label{th: alagic}
    Consider the problem $\mathcal{P}$ of determining a value-distinguishing approximation of the $SO_r(3)$-RT invariant, $r\geq 5$ prime, of a manifold $M$, represented through a Heegaard splitting described by a word $\phi\in \Mod(\Sigma_g)$ for some known $g\geq 2$. Then $\mathcal{P}$ is \#\texttt{P}-hard in the sense of a Cook-Turing reduction.
\end{theorem}

Before applying Theorem \ref{th: properties} to this result, we need to find an algorithm to convert the pair $(\Sigma_g,\phi)$ into a proper Heegaard diagram $(\Sigma_g,\beta)$. 
This cannot be done through brute force computing $\beta=\phi(\alpha)$, since, as we saw in the proof of Proposition \ref{prop: launching from diskbusting}, it can lead to exponential bottlenecks.
We fix the problem with the following lemma, at the cost of potentially increasing the value of $g$. 

\begin{lemma}\label{lm: word to diagram}
    Consider a Heegaard splitting described by a word $\phi\in \Mod(\Sigma_g)$ for some known $g\geq 2$. Then it is possible to compute, in time $O(\poly(g,\phi))$, a Heegaard diagram $(\Sigma_{g'},\beta)$ representing the same manifold, with $\beta$ normal to a triangulation $T$ of $\Sigma_{g'}$.
\end{lemma}
\begin{proof}
    Given a splitting $(\Sigma_g,\phi)$, a triangulation $\mathcal{T}$ for $M$ (refer to the Section \ref{sec: turaev-viro} for a quick introduction of triangulations of 3-manifolds) can be determined in time $O(g|\phi|)$ and with as many tetrahedra \cite{alagic2011quantum, alagic2014quantum, brinkmann2001computing,koenig2010quantum}.
    The standard proof of the existence of a Heegaard splitting for any closed 3-manifold from its triangulation \cite[Theorem 2.5]{hempel20223} is constructive, and can be made algorithmic to produce a Heegaard diagram $(\Sigma_{g'},\beta)$ of the corresponding 3-manifold, where $\Sigma_{g'}$ is the boundary of the standard embedding of the genus $g'$ handlebody in $\mathbb{R}^3$ represented by a triangulation $\mathcal{T}'$. 
    The 3-manifold triangulation $\mathcal{T}'$ naturally induces a surface triangulation $T$ on $\Sigma_{g'}$.
    By construction, the curves $\beta$ are represented as an edge list in $T$, so to make them normal, we use Proposition \ref{prop: triangulation to normal}.
\end{proof}

\begin{corollary}\label{col: RT}
     Fix a prime $r\geq 5$. Consider the problem $\mathcal{P}$ of, given a Heegaard diagram $(\Sigma_g,\beta)$ of a closed 3-manifold $M$, returning a value-distinguishing approximation of its $SO_r(3)$-RT invariant if $M$ has the properties of Theorem \ref{th: properties}, otherwise remaining silent. Then $\mathcal{P}$ is \#\texttt{P}-hard in the sense of a Cook-Turing reduction. 
\end{corollary}

\begin{proof}
    Let $\mathcal{O}$ be an oracle machine that solves $\mathcal{P}$ and consider the problem $\mathcal{P}'$ of finding a value-distinguishing approximation of a {general Heegaard splitting} given as a pair $(\Sigma_g,\phi)$ for $\phi \in \Mod(\Sigma_g)$. 
    We will show that $\mathcal{O}$ solves $P'$ with only a polynomial overhead. 
    Because $\mathcal{O}$ solves $\mathcal{P}$ and, by Theorem \ref{th: alagic}, $\mathcal{P}'$ is \#\texttt{P}-hard, then so is $\mathcal{P}$.
    
    Let $(\Sigma_g,\phi)$ encode a Heegaard splitting of a manifold $M$, not necessarily with the properties of Theorem \ref{th: properties}. 
    Using Lemma \ref{lm: word to diagram}, we transform the data $(\Sigma_g,\phi)$ into a Heegaard diagram $(\Sigma_{g'},\beta)$ in polynomial time. 
    Then apply the algorithm of Theorem \ref{th: properties} to this diagram, returning three new diagrams as output.
    We run the oracle $\mathcal{O}$ in parallel to these three diagrams, stopping the program whenever the oracle halts for one of them.
    In the end, this gives, in polynomial time, a value-distinguishing approximation of $\langle M \rangle^{RT}_{SO_r(3)}$, concluding the proof.
\end{proof}
\begin{remark}
Reshetikhin–Turaev's original presentation of the invariant \cite{turaev2010quantum}, instead of computed from a Heegaard splitting, is combinatorially defined through the so-called \textit{surgery presentation of 3-manifolds}, in which every framed link in $S^3$ is associated with a closed manifold $M$. 
Careful analysis of Lickorish's \cite{lickorish1962representation} original proof of the existence of the surgery presentation indicates that, given a pair $(\Sigma_g,\phi)$, $\phi\in \Mod(\Sigma_g)$, one can compute, in polynomial time in $|\phi|$, a surgery presentation of the 3-manifold. 
This implies that the complexity of approximating the RT-invariant of a general 3-manifold taking $(\Sigma_g,\phi)$ as input carries over to the problem of approximating the invariant taking a framed link as input.
To the best of our knowledge, there is yet no efficient algorithm to transform a Heegaard diagram $(\Sigma_g,\beta)$ into a surgery presentation.
In consequence, the complexity of approximating the RT-invariant of a manifold with the properties of Theorem \ref{th: properties} given by a framed link diagram remains therefore open.
\end{remark}
\begin{remark}
    We note that the sort of reduction assumed by the statement of Corollary \ref{col: RT} is related to what is often called in the complexity literature by a \emph{semi-decision problem}, that is, an oracle that cannot return an incorrect answer, but may not halt for some inputs. Although weaker than the more common approach in which we assume that $\mathcal{P}$ \emph{can} give incorrect approximation of the invariant if the input does not have the expected properties, this sort of oracle has also already been discussed for algorithm on 3-manifolds, e.g. see the Definition 1.3 in \cite{manning2002algorithmic}.
\end{remark}

\subsection{Turaev-Viro invariant}\label{sec: turaev-viro}
A compact 3-manifold can also be combinatorially described by a set of tetrahedra $\mathcal{T}$, together with rules on how to glue their triangular faces \cite{bing1959alternative, moise1952affine}. 
This description is called a \textit{triangulation of the 3-manifold}, but it should not be confused with triangulations of surfaces. 
Nonetheless, if $M$ has a boundary, $\mathcal{T}$ naturally defines a (surface) triangulation for $\partial M$. 
We note that for each $g\geq 1$, there exists a one-vertex triangulation of the handlebody of genus $g$ \cite{jaco2006layeredtriangulations3manifolds}. 

The Turaev-Viro invariant (TV) is another quantum invariant for closed 3-manifolds.
It is defined for spherical categories (which include modular categories) and is computed directly from a triangulation \cite{barrett1996invariants}. 
The Turaev-Walker theorem \cite{turaev2010quantum} states that, given a manifold $M$, $|\langle M \rangle_{\mathcal{C}}^{RT}|^2=\langle M \rangle_{\mathcal{C}}^{TV}$, provided $\mathcal{C}$ is a modular category.
An approximation of the $SO_r(3)$-TV invariant can be used to compute an approximation of $SO_r(3)$-RT \cite{alagic2014quantum}, as long as there is a polynomial time algorithm to transform a Heegaard splitting into a triangulation. 
For such, we use the next theorem, whose proof, except for the complexity analysis below, is due to \cite{he2023algorithm}.

\begin{theorem}\label{th: glue handlebody}
    Suppose $\beta$ is a normal minimal system given with respect to a one-vertex triangulation of the standard embedding of the genus $g$ handlebody in $\mathbb{R}^3$, with $\|I_T(\beta)\|=m$. There is a deterministic algorithm to compute, in time $\poly(m,g)$, a triangulation of the 3-manifold of Heegaard diagram $(\Sigma_g,\beta)$.
\end{theorem}
\begin{proof}
    The algorithm consists of some subroutines whose correct convergences are proved by \cite{he2023algorithm}. 
    Direct computations give the complexity of each subroutine to be
    \begin{itemize}
        \item Proposition 13: $O(m^2+g^2)$, where we use our Proposition \ref{prop: normal to triangulation} to trace $\beta$;
        \item \texttt{Petal-resolver}: $O(mg)$;
        \item  \texttt{Quad-isolator}: $O(g^2)$
        \item \texttt{Wedge-folder}: $O(g)$;
        \item \texttt{Ball-filler}: $O(g^2)$.
    \end{itemize}
\end{proof}

\begin{corollary}\label{col: TV}
    Fix a prime $r\geq 5$. Consider the problem $\mathcal{P}$ of, given a triangulation of a closed 3-manifold $M$, returning a value-distinguishing approximation of its $SO_r(3)$-TV invariant if $M$ has the properties of Theorem \ref{th: properties}, otherwise remaining silent. 
    Then $\mathcal{P}$ is \#\texttt{P}-hard in the sense of Cook-Turing reduction. 
\end{corollary}

\begin{proof}
    An oracle $\mathcal{O}$ that approximates $SO_r(3)$-TV can be used to approximate $SO_r(3)$-RT for any prime $r\geq 5$ \cite[Theorem 5.1]{alagic2014quantum}. 
    Moreover, we note that one can transform a general triangulation of the standard embedding of the genus $g$ handlebody with $\beta$ curves represented as intersection words into a one-vertex triangulation in time $O(\poly(g,\|I_T(\beta)\|,|T|))$ using standard computational topology arguments. 
    More precisely, Jaco and Rubinstein introduce and analyze in~\cite{jaco20030} the operation of {\em crushing} a normal surface in a possibly bounded triangulated 3-manifold, which can be implemented by a polynomial time algorithm (see also~\cite{Burton14}). 
    For any triangulation of a 3-manifold with at least two vertices (except for an exceptional behavior with triangulations of $S^3$), they prove the existence of a normal sphere (or normal disk) bounding a ball, whose crushing reduces the number of vertices while preserving the topology (see~\cite[Proof of Proposition 5.1]{jaco20030} for the existence of the normal sphere/disk, and~\cite[Theorem 2]{Burton14}, originally proved in~\cite{jaco20030}, for the preservation of the topology under crushing). 
    This sphere can be found in polynomial time by considering the link of an edge with distinct endpoints, and normalizing it~\cite[Proof of Proposition 5.1]{jaco20030}.
    It follows that, as long as the triangulation has more than two vertices, one can find and crush in polynomial time a normal sphere or disk bounding a 3-ball containing an edge connecting distinct vertices, and reduce the number of vertices while preserving the topology. 
    Consequently, by Theorem \ref{th: glue handlebody}, one can reduce the problem of deciding on a value-distinguishing approximation of the RT invariant to deciding on a value-distinguishing approximation of the TV invariant. 
    Because, by Corollary \ref{col: RT}, the former is \#\texttt{P}-hard, the latter will also be.
\end{proof}
\bibliographystyle{plainurl}
\bibliography{hardness}

\begin{thebibliography}{10}

\bibitem{agol2006computational}
Ian Agol, Joel Hass, and William Thurston.
\newblock The computational complexity of knot genus and spanning area.
\newblock {\em Transactions of the American Mathematical Society}, 358(9):3821--3850, 2006.
\newblock URL: \url{https://arxiv.org/abs/math/0205057}.

\bibitem{aharonovJL2009polynomial}
D.~Aharonov, V.~Jones, and Z~Landau.
\newblock A polynomial quantum algorithm for approximating the {J}ones polynomial.
\newblock {\em Algorithmica}, 55:395--421, 2009.
\newblock \href {https://doi.org/10.1007/s00453-008-9168-0} {\path{doi:10.1007/s00453-008-9168-0}}.

\bibitem{alagic2011quantum}
Gorjan Alagic and Edgar~A. Bering.
\newblock {Quantum algorithms for invariants of triangulated manifolds}.
\newblock {\em Quant. Inf. Comput.}, 12(9-10):0843--0863, 2012.
\newblock \href {https://doi.org/10.26421/QIC12.9-10-8} {\path{doi:10.26421/QIC12.9-10-8}}.

\bibitem{alagic2014quantum}
Gorjan Alagic and Catharine Lo.
\newblock Quantum invariants of 3-manifolds and {NP} vs \#{P}.
\newblock {\em Quantum Info. Comput.}, 17(1–2):125–146, feb 2017.
\newblock URL: \url{https://arxiv.org/abs/1411.6049}.

\bibitem{aradL2010quantum}
Itai Arad and Zeph Landau.
\newblock Quantum computation and the evaluation of tensor networks.
\newblock {\em SIAM Journal on Computing}, 39(7):3089--3121, 2010.
\newblock \href {https://doi.org/10.1137/080739379} {\path{doi:10.1137/080739379}}.

\bibitem{Arouca_2022}
Rodrigo Arouca, Andrea Cappelli, and Hans Hansson.
\newblock Quantum field theory anomalies in condensed matter physics.
\newblock {\em SciPost Physics Lecture Notes}, sep 2022.
\newblock URL: \url{https://arxiv.org/abs/2204.02158}, \href {https://doi.org/10.21468/scipostphyslectnotes.62} {\path{doi:10.21468/scipostphyslectnotes.62}}.

\bibitem{bachman2017computing}
David Bachman, Ryan Derby-Talbot, and Eric Sedgwick.
\newblock Computing {H}eegaard genus is {NP}-hard.
\newblock {\em A Journey Through Discrete Mathematics: A Tribute to Ji{\v{r}}{\'\i} Matou{\v{s}}ek}, pages 59--87, 2017.
\newblock URL: \url{https://arxiv.org/abs/1606.01553}.

\bibitem{barrett1996invariants}
John Barrett and Bruce Westbury.
\newblock Invariants of piecewise-linear 3-manifolds.
\newblock {\em Transactions of the American Mathematical Society}, 348(10):3997--4022, 1996.
\newblock URL: \url{https://arxiv.org/abs/hep-th/9311155}.

\bibitem{bell2016polynomial}
Mark~C. Bell and Richard C.~H. Webb.
\newblock Polynomial-time algorithms for the curve graph.
\newblock {\em arXiv: Geometric Topology}, 2016.
\newblock URL: \url{https://arxiv.org/abs/1609.09392}.

\bibitem{bell2015recognising}
Mark~Christopher Bell.
\newblock {\em Recognising mapping classes}.
\newblock PhD thesis, University of Warwick, 2015.
\newblock URL: \url{https://wrap.warwick.ac.uk/id/eprint/77123/}.

\bibitem{BenedettiPetronioHyperbolicGeom}
Riccardo Benedetti and Carlo Petronio.
\newblock {\em Lectures on Hyperbolic Geometry}.
\newblock Springer, 1992.

\bibitem{bing1959alternative}
RH~Bing.
\newblock An alternative proof that 3-manifolds can be triangulated.
\newblock {\em Annals of Mathematics}, 69(1):37--65, 1959.
\newblock \href {https://doi.org/10.2307/1970092} {\path{doi:10.2307/1970092}}.

\bibitem{brinkmann2001computing}
Peter Brinkmann and Saul Schleimer.
\newblock Computing triangulations of mapping tori of surface homeomorphisms.
\newblock {\em Experimental Mathematics}, 10(4):571--581, 2001.
\newblock \href {https://doi.org/10.1080/10586458.2001.10504677} {\path{doi:10.1080/10586458.2001.10504677}}.

\bibitem{Burton14}
Benjamin~A. Burton.
\newblock A new approach to crushing 3-manifold triangulations.
\newblock {\em Discrete \& Computational Geometry}, 52:116--139, 2014.
\newblock \href {https://doi.org/10.1007/s00454-014-9572-y} {\path{doi:10.1007/s00454-014-9572-y}}.

\bibitem{burton:LIPIcs.SoCG.2018.18}
Benjamin~A. Burton.
\newblock {The HOMFLY-PT Polynomial is Fixed-Parameter Tractable}.
\newblock In Bettina Speckmann and Csaba~D. T\'{o}th, editors, {\em 34th International Symposium on Computational Geometry (SoCG 2018)}, volume~99 of {\em Leibniz International Proceedings in Informatics (LIPIcs)}, pages 18:1--18:14, Dagstuhl, Germany, 2018. Schloss Dagstuhl -- Leibniz-Zentrum f{\"u}r Informatik.
\newblock \href {https://doi.org/10.4230/LIPIcs.SoCG.2018.18} {\path{doi:10.4230/LIPIcs.SoCG.2018.18}}.

\bibitem{burton:LIPIcs.SoCG.2020.25}
Benjamin~A. Burton.
\newblock {The Next 350 Million Knots}.
\newblock In Sergio Cabello and Danny~Z. Chen, editors, {\em 36th International Symposium on Computational Geometry (SoCG 2020)}, volume 164 of {\em Leibniz International Proceedings in Informatics (LIPIcs)}, pages 25:1--25:17, Dagstuhl, Germany, 2020. Schloss Dagstuhl -- Leibniz-Zentrum f{\"u}r Informatik.
\newblock \href {https://doi.org/10.4230/LIPIcs.SoCG.2020.25} {\path{doi:10.4230/LIPIcs.SoCG.2020.25}}.

\bibitem{BurtonMS18}
Benjamin~A. Burton, Cl{\'{e}}ment Maria, and Jonathan Spreer.
\newblock Algorithms and complexity for {T}uraev-{V}iro invariants.
\newblock {\em Journal of Applied and Computational Topology}, 2(1-2):33--53, 2018.
\newblock \href {https://doi.org/10.1007/s41468-018-0016-2} {\path{doi:10.1007/s41468-018-0016-2}}.

\bibitem{chen2018quantum}
Qingtao Chen and Tian Yang.
\newblock Volume conjectures for the {R}eshetikhin–{T}uraev and the {T}uraev–{V}iro invariants.
\newblock {\em Quantum Topology}, 3:419--460, 2018.
\newblock URL: \url{https://arxiv.org/abs/1503.02547}.

\bibitem{delaney2024algorithmtambarayamagamiquantuminvariants}
Colleen Delaney, Clément Maria, and Eric Samperton.
\newblock An algorithm for {T}ambara-{Y}amagami quantum invariants of 3-manifolds, parameterized by the first {B}etti number, 2024.
\newblock URL: \url{https://arxiv.org/html/2311.08514v2}, \href {https://arxiv.org/abs/2311.08514} {\path{arXiv:2311.08514}}.

\bibitem{erickson2012tracing}
Jeff Erickson and Amir Nayyeri.
\newblock Tracing compressed curves in triangulated surfaces.
\newblock In {\em Proceedings of the twenty-eighth annual symposium on Computational geometry}, pages 131--140, 2012.
\newblock \href {https://doi.org/10.1145/2261250.2261270} {\path{doi:10.1145/2261250.2261270}}.

\bibitem{etingof2002vafa}
Pavel Etingof.
\newblock On {V}afa’s theorem for tensor categories.
\newblock {\em Mathematical Research Letters}, 9:651--657, 2002.
\newblock URL: \url{https://arxiv.org/abs/math/0207007}.

\bibitem{evans2006high}
Tatiana Evans.
\newblock High distance {H}eegaard splittings of 3-manifolds.
\newblock {\em Topology and its Applications}, 153(14):2631--2647, 2006.
\newblock \href {https://doi.org/10.1016/j.topol.2005.11.003} {\path{doi:10.1016/j.topol.2005.11.003}}.

\bibitem{farb2011primer}
Benson Farb and Dan Margalit.
\newblock {\em A primer on mapping class groups}, volume~41.
\newblock Princeton University Press, 2011.

\bibitem{Fradkin_2024}
Eduardo Fradkin.
\newblock {\em Field theoretic aspects of condensed matter physics: An overview}, page 27–131.
\newblock Elsevier, 2024.
\newblock URL: \url{https://arxiv.org/abs/2301.13234}, \href {https://doi.org/10.1016/b978-0-323-90800-9.00269-9} {\path{doi:10.1016/b978-0-323-90800-9.00269-9}}.

\bibitem{freedmanLW2002topological}
M.~Freedman, M.~Larsen, and Z.~Wang.
\newblock The two-eigenvalue problem and density of {J}ones representation of braid groups.
\newblock {\em Communications in Mathematical Physics}, 228:177--199, 2002.
\newblock URL: \url{https://arxiv.org/abs/math/0103200}.

\bibitem{haken1961theorie}
Wolfgang Haken.
\newblock Theorie der normalfl{\"a}chen: ein isotopiekriterium f{\"u}r den kreisknoten.
\newblock 1961.

\bibitem{hartshorn2002heegaard}
Kevin Hartshorn.
\newblock Heegaard splittings of haken manifolds have bounded distance.
\newblock {\em Pacific journal of mathematics}, 204(1):61--75, 2002.
\newblock URL: \url{https://msp.org/pjm/2002/204-1/p05.xhtml}.

\bibitem{hatcher2002algebraic}
Allen Hatcher.
\newblock {\em Algebraic Topology}.
\newblock Cambridge University Press, 2002.

\bibitem{hatcher1980presentation}
Allen Hatcher and William Thurston.
\newblock A presentation for the mapping class group of a closed orientable surface.
\newblock {\em Topology}, 19(3):221--237, 1980.
\newblock URL: \url{https://pi.math.cornell.edu/~hatcher/Papers/MCGpresentation.pdf}.

\bibitem{he2023algorithm}
Alexander He, James Morgan, and Em~K Thompson.
\newblock An algorithm to construct one-vertex triangulations of {H}eegaard splittings.
\newblock {\em arXiv preprint arXiv:2312.17556}, 2023.
\newblock URL: \url{https://arxiv.org/abs/2312.17556}.

\bibitem{hempel20013}
John Hempel.
\newblock 3-manifolds as viewed from the curve complex.
\newblock {\em Topology}, 40(3):631--657, 2001.
\newblock URL: \url{https://arxiv.org/abs/math/9712220}.

\bibitem{hempel20223}
John Hempel.
\newblock {\em 3-Manifolds}, volume 349.
\newblock American Mathematical Society, 2022.

\bibitem{huszar_et_al:LIPIcs.SoCG.2019.44}
Krist\'{o}f Husz\'{a}r and Jonathan Spreer.
\newblock {3-Manifold Triangulations with Small Treewidth}.
\newblock In Gill Barequet and Yusu Wang, editors, {\em 35th International Symposium on Computational Geometry (SoCG 2019)}, volume 129 of {\em Leibniz International Proceedings in Informatics (LIPIcs)}, pages 44:1--44:20, Dagstuhl, Germany, 2019. Schloss Dagstuhl -- Leibniz-Zentrum f{\"u}r Informatik.
\newblock \href {https://doi.org/10.4230/LIPIcs.SoCG.2019.44} {\path{doi:10.4230/LIPIcs.SoCG.2019.44}}.

\bibitem{huszar_et_al:LIPIcs.SoCG.2023.42}
Krist\'{o}f Husz\'{a}r and Jonathan Spreer.
\newblock {On the Width of Complicated JSJ Decompositions}.
\newblock In Erin~W. Chambers and Joachim Gudmundsson, editors, {\em 39th International Symposium on Computational Geometry (SoCG 2023)}, volume 258 of {\em Leibniz International Proceedings in Informatics (LIPIcs)}, pages 42:1--42:18, Dagstuhl, Germany, 2023. Schloss Dagstuhl -- Leibniz-Zentrum f{\"u}r Informatik.
\newblock \href {https://doi.org/10.4230/LIPIcs.SoCG.2023.42} {\path{doi:10.4230/LIPIcs.SoCG.2023.42}}.

\bibitem{Huszár_2022}
Kristóf Huszár.
\newblock On the pathwidth of hyperbolic 3-manifolds.
\newblock {\em Computing in Geometry and Topology}, 1(1):1:1–1:19, Feb. 2022.
\newblock URL: \url{https://arxiv.org/abs/2105.11371}, \href {https://doi.org/10.57717/cgt.v1i1.4} {\path{doi:10.57717/cgt.v1i1.4}}.

\bibitem{ido2014heegaard}
Ayako Ido, Yeonhee Jang, and Tsuyoshi Kobayashi.
\newblock Heegaard splittings of distance exactly n.
\newblock {\em Algebraic \& geometric topology}, 14(3):1395--1411, 2014.
\newblock URL: \url{https://arxiv.org/abs/1210.7627}.

\bibitem{jaco20030}
William Jaco and J~Hyam Rubinstein.
\newblock 0-efficient triangulations of 3-manifolds.
\newblock {\em Journal of Differential Geometry}, 65(1):61--168, 2003.
\newblock URL: \url{https://arxiv.org/abs/math/0207158}.

\bibitem{jaco2006layeredtriangulations3manifolds}
William Jaco and J.~Hyam Rubinstein.
\newblock Layered-triangulations of 3-manifolds, 2006.
\newblock URL: \url{https://arxiv.org/abs/math/0603601}, \href {https://arxiv.org/abs/math/0603601} {\path{arXiv:math/0603601}}.

\bibitem{johnson2006notes}
Jesse Johnson.
\newblock Notes on {H}eegaard splittings.
\newblock {\em preprint}, 2006.

\bibitem{johnson2013non}
Jesse Johnson.
\newblock Non-uniqueness of high distance {H}eegaard splittings.
\newblock {\em arXiv preprint arXiv:1308.4599}, 2013.
\newblock URL: \url{https://arxiv.org/abs/1308.4599}.

\bibitem{KITAEV20032}
A.Yu. Kitaev.
\newblock Fault-tolerant quantum computation by anyons.
\newblock {\em Annals of Physics}, 303(1):2--30, 2003.
\newblock URL: \url{https://arxiv.org/abs/quant-ph/9707021}, \href {https://doi.org/10.1016/S0003-4916(02)00018-0} {\path{doi:10.1016/S0003-4916(02)00018-0}}.

\bibitem{kobayashi1988casson}
Tsuyoshi Kobayashi.
\newblock Casson-{G}ordon's rectangle condition of {H}eegaard diagrams and incompressible tori in 3-manifolds.
\newblock {\em Osaka Journal of Mathematics}, 25:553--573, 1988.

\bibitem{koenig2010quantum}
Robert Koenig, Greg Kuperberg, and Ben~W Reichardt.
\newblock Quantum computation with {T}uraev--{V}iro codes.
\newblock {\em Annals of Physics}, 325(12):2707--2749, 2010.
\newblock URL: \url{https://arxiv.org/abs/1002.2816}.

\bibitem{kuperberg2009hard}
Greg Kuperberg.
\newblock How hard is it to approximate the {J}ones {P}olynomial?
\newblock {\em Theory Computing}, 11:183--219, 2009.
\newblock URL: \url{https://arxiv.org/abs/0908.0512}.

\bibitem{kuperberg2019algorithmic}
Greg Kuperberg.
\newblock Algorithmic homeomorphism of 3-manifolds as a corollary of geometrization.
\newblock {\em Pacific Journal of Mathematics}, 301(1):189--241, 2019.
\newblock URL: \url{https://arxiv.org/abs/1508.06720}.

\bibitem{lickorish1962representation}
WB~Raymond Lickorish.
\newblock A representation of orientable combinatorial 3-manifolds.
\newblock {\em Annals of Mathematics}, 76(3):531--540, 1962.
\newblock \href {https://doi.org/10.2307/1970373} {\path{doi:10.2307/1970373}}.

\bibitem{lickorish1964finite}
William~BR Lickorish.
\newblock A finite set of generators for the homeotopy group of a 2-manifold.
\newblock In {\em Mathematical Proceedings of the Cambridge Philosophical Society}, volume~60, pages 769--778. Cambridge University Press, 1964.
\newblock \href {https://doi.org/10.1017/S030500410003824X} {\path{doi:10.1017/S030500410003824X}}.

\bibitem{lustig2009high}
Martin Lustig and Yoav Moriah.
\newblock High distance {H}eegaard splittings via fat train tracks.
\newblock {\em Topology and its Applications}, 156(6):1118--1129, 2009.
\newblock URL: \url{https://arxiv.org/abs/0706.0599}.

\bibitem{MAKOWSKY2003742}
J.A. Makowsky and J.P. Mariño.
\newblock The parametrized complexity of knot polynomials.
\newblock {\em Journal of Computer and System Sciences}, 67(4):742--756, 2003.
\newblock Parameterized Computation and Complexity 2003.
\newblock \href {https://doi.org/10.1016/S0022-0000(03)00080-1} {\path{doi:10.1016/S0022-0000(03)00080-1}}.

\bibitem{manning2002algorithmic}
Jason~Fox Manning.
\newblock Algorithmic detection and description of hyperbolic structures on closed 3--manifolds with solvable word problem.
\newblock {\em Geometry \& Topology}, 6(1):1--26, 2002.
\newblock URL: \url{https://arxiv.org/abs/math/0102154}.

\bibitem{Maria21}
Cl{\'{e}}ment Maria.
\newblock Parameterized complexity of quantum knot invariants.
\newblock In {\em Proceedings of the International Symposium on Computational Geometry, {SoCG}}, volume 189 of {\em LIPIcs}, pages 53:1--53:17, 2021.
\newblock \href {https://doi.org/10.4230/LIPIcs.SoCG.2021.53} {\path{doi:10.4230/LIPIcs.SoCG.2021.53}}.

\bibitem{MariaP19}
Cl{\'{e}}ment Maria and Jessica~S. Purcell.
\newblock Treewidth, crushing and hyperbolic volume.
\newblock {\em Algebraic \& Geometric Topology}, 19:2625–2652, 2019.
\newblock URL: \url{https://arxiv.org/abs/1805.02357}.

\bibitem{MariaR21}
Cl{\'{e}}ment Maria and Owen Rouill{\'{e}}.
\newblock Computation of large asymptotics of 3-manifold quantum invariants.
\newblock In {\em Proceedings of the Symposium on Algorithm Engineering and Experiments, {ALENEX}}, pages 193--206. {SIAM}, 2021.
\newblock \href {https://doi.org/10.1137/1.9781611976472.15} {\path{doi:10.1137/1.9781611976472.15}}.

\bibitem{MariaS20}
Cl{\'{e}}ment Maria and Jonathan Spreer.
\newblock A polynomial-time algorithm to compute {T}uraev-{V}iro invariants tv(4,q) of 3-manifolds with bounded first {B}etti number.
\newblock {\em Foundations of Computational Mathematics}, 20(5):1013--1034, 2020.
\newblock URL: \url{https://arxiv.org/abs/1607.02218}.

\bibitem{matveev03-algms}
Sergei Matveev.
\newblock {\em Algorithmic Topology and Classification of 3-Manifolds}.
\newblock Number~9 in Algorithms and Computation in Mathematics. Springer, Berlin, 2003.

\bibitem{moise1952affine}
Edwin~E Moise.
\newblock Affine structures in 3-manifolds: V. the triangulation theorem and hauptvermutung.
\newblock {\em Annals of mathematics}, 56(1):96--114, 1952.
\newblock \href {https://doi.org/10.2307/1969769} {\path{doi:10.2307/1969769}}.

\bibitem{muller2023dehn}
Lukas M{\"u}ller and Lukas Woike.
\newblock The {D}ehn twist action for quantum representations of mapping class groups.
\newblock {\em arXiv preprint arXiv:2311.16020}, 2023.
\newblock URL: \url{https://arxiv.org/abs/2311.16020}.

\bibitem{perelman2003finiteextinctiontimesolutions}
Grisha Perelman.
\newblock Finite extinction time for the solutions to the {R}icci flow on certain three-manifolds, 2003.
\newblock URL: \url{https://arxiv.org/abs/math/0307245}, \href {https://arxiv.org/abs/math/0307245} {\path{arXiv:math/0307245}}.

\bibitem{perelman2003ricciflowsurgerythreemanifolds}
Grisha Perelman.
\newblock Ricci flow with surgery on three-manifolds, 2003.
\newblock URL: \url{https://www.arxiv.org/abs/math/0303109}, \href {https://arxiv.org/abs/math/0303109} {\path{arXiv:math/0303109}}.

\bibitem{qiu2015heegaard}
Ruifeng Qiu, Yanqing Zou, and Qilong Guo.
\newblock The {H}eegaard distances cover all nonnegative integers.
\newblock {\em Pacific Journal of Mathematics}, 275(1):231--255, 2015.
\newblock URL: \url{https://arxiv.org/abs/1302.5188}.

\bibitem{salvatore2001frameddiscsoperadsequivariant}
Paolo Salvatore and Nathalie Wahl.
\newblock Framed discs operads and the equivariant recognition principle, 2001.
\newblock URL: \url{https://arxiv.org/abs/math/0106242}, \href {https://arxiv.org/abs/math/0106242} {\path{arXiv:math/0106242}}.

\bibitem{samperton2023topological}
Eric Samperton.
\newblock Topological quantum computation is hyperbolic.
\newblock {\em Communications in Mathematical Physics}, 402:79--96, 2023.
\newblock URL: \url{https://arxiv.org/abs/2201.00857}.

\bibitem{saveliev2011lectures}
Nikolai Saveliev.
\newblock {\em Lectures on the topology of 3-manifolds: an introduction to the Casson invariant}.
\newblock Walter de Gruyter, 2011.

\bibitem{schaefer2002algorithms}
Marcus Schaefer, Eric Sedgwick, and Daniel {\v{S}}tefankovi{\v{c}}.
\newblock Algorithms for normal curves and surfaces.
\newblock In {\em Computing and Combinatorics: 8th Annual International Conference, COCOON 2002 Singapore, August 15--17, 2002 Proceedings 8}, pages 370--380. Springer, 2002.
\newblock \href {https://doi.org/10.1007/3-540-45655-4_40} {\path{doi:10.1007/3-540-45655-4_40}}.

\bibitem{schaefer2008computing}
Marcus Schaefer, Eric Sedgwick, and Daniel Stefankovic.
\newblock Computing {D}ehn twists and geometric intersection numbers in polynomial time.
\newblock In {\em CCCG}, volume~20, pages 111--114, 2008.
\newblock URL: \url{https://www.cs.rochester.edu/~stefanko/Publications/geometric.pdf}.

\bibitem{scharlemann2006alternate}
Martin Scharlemann and Maggy Tomova.
\newblock Alternate {H}eegaard genus bounds distance.
\newblock {\em Geometry \& Topology}, 10(1):593--617, 2006.
\newblock URL: \url{https://arxiv.org/abs/math/0501140}.

\bibitem{schultens2014introduction}
Jennifer Schultens.
\newblock {\em Introduction to 3-manifolds}, volume 151.
\newblock American Mathematical Soc., 2014.

\bibitem{scull2021homeomorphismproblemhyperbolicmanifolds}
Joe Scull.
\newblock The homeomorphism problem for hyperbolic manifolds {I}, 2021.
\newblock URL: \url{https://arxiv.org/abs/2108.00779}, \href {https://arxiv.org/abs/2108.00779} {\path{arXiv:2108.00779}}.

\bibitem{singer2015lecture}
Isadore~Manuel Singer and John~A Thorpe.
\newblock {\em Lecture notes on elementary topology and geometry}.
\newblock Springer, 2015.

\bibitem{starr1992curves}
Edith~Nelson Starr.
\newblock {\em Curves in handlebodies}.
\newblock PhD thesis, University of California, Berkeley, 1992.

\bibitem{thompson1999disjoint}
Abigail Thompson.
\newblock The disjoint curve property and genus 2 manifolds.
\newblock {\em Topology and its Applications}, 97(3):273--279, 1999.

\bibitem{turaev2010quantum}
Vladimir~G Turaev.
\newblock {\em Quantum invariants of knots and 3-manifolds}.
\newblock de Gruyter, 2010.

\bibitem{vafa1988toward}
Cumrun Vafa.
\newblock Toward classification of conformal theories.
\newblock {\em Physics Letters B}, 206(3):421--426, 1988.

\bibitem{verdiere2007optimal}
{\'E}ric Colin~De Verdi{\`e}re and Francis Lazarus.
\newblock Optimal pants decompositions and shortest homotopic cycles on an orientable surface.
\newblock {\em Journal of the ACM (JACM)}, 54(4):18--es, 2007.
\newblock \href {https://doi.org/10.1007/978-3-540-24595-7_45} {\path{doi:10.1007/978-3-540-24595-7_45}}.

\bibitem{yoshizawa2014high}
Michael Yoshizawa.
\newblock High distance {H}eegaard splittings via {D}ehn twists.
\newblock {\em Algebraic \& Geometric Topology}, 14(2):979--1004, 2014.
\newblock URL: \url{https://arxiv.org/abs/1212.1199}.

\end{thebibliography}

\end{document}